\documentclass[12pt,letterpaper,reqno]{amsart}

\usepackage{times}
\usepackage[T1]{fontenc}
\usepackage{mathrsfs}
\usepackage{latexsym}
\usepackage[dvips]{graphics}
\usepackage{epsfig}
\usepackage{amsmath,amsfonts,amsthm,amssymb,amscd}
\input amssym.def
\input amssym.tex

\newcommand{\supp}{{\rm supp}}

\newcommand{\hg}{\widehat{g}}

\newcommand{\hphi}{\widehat{\phi}}  %phi^

\newcommand{\FD}{\mathcal{F}} %fundamental domain

\newcommand\be{\begin{equation}}
\newcommand\ee{\end{equation}}
\newcommand\bea{\begin{eqnarray}}
\newcommand\eea{\end{eqnarray}}
\newcommand\bi{\begin{itemize}}
\newcommand\ei{\end{itemize}}
\newcommand\ben{\begin{enumerate}}
\newcommand\een{\end{enumerate}}
\newcommand\bc{\begin{center}}
\newcommand\ec{\end{center}}
\newcommand\ba{\begin{array}}
\newcommand\ea{\end{array}}
\def\notdiv{\ \mathbin{\mkern-8mu|\!\!\!\smallsetminus}}
\newcommand{\R}{\ensuremath{\mathbb{R}}}

\newcommand{\Q}{\mathbb{Q}}

\newcommand{\foh}{\frac{1}{2}}  %onehalf
\newtheorem{thm}{Theorem}[section]

\newtheorem{lem}[thm]{Lemma}

\theoremstyle{definition}
\newtheorem{rek}[thm]{Remark}

\newcommand{\twocase}[5]{#1 \begin{cases} #2 & \text{{\rm #3}}\\ #4
&\text{{\rm #5}} \end{cases}   }

\newcommand{\gep}{\epsilon}
\newcommand{\gl}{\lambda}
\newcommand{\ga}{\alpha}

\numberwithin{equation}{section}

\begin{document}

\title[An Orthogonal Test of the $L$-Functions Ratios Conjecture]{An
Orthogonal Test of the $L$-Functions Ratios Conjecture}

\author{Steven J. Miller}
\email{Steven.J.Miller@williams.edu} \address{\noindent Department of Mathematics, Brown University, Providence, RI 02912 and Department of Mathematics and Statistics,
Williams College, Williamstown, MA 02167}

\subjclass[2000]{11M26 (primary), 11M41, 15A52 (secondary).}
\keywords{$1$-level density, cuspidal newforms, low lying
zeros, Ratios Conjecture}

\date{\today}

\thanks{We thank Eduardo Due$\tilde{{\rm n}}$ez, David Farmer, Duc Khiem Huynh, Jon Keating and Nina Snaith for many enlightening conversations. Some of this work was carried out while the author was visiting the University of Bristol, and it is a pleasure to thank them for their hospitality. This work was partly
supported by NSF grant DMS0600848.}

\begin{abstract} We test the predictions of the $L$-functions Ratios Conjecture for the family of cuspidal newforms of weight $k$ and level $N$, with either $k$ fixed and $N\to\infty$ through the primes or $N=1$ and $k\to\infty$. We study the main and lower order terms in the 1-level density. We provide evidence for the Ratios Conjecture by computing and confirming its predictions up to a power savings in the family's cardinality, at least for test functions whose Fourier transforms are supported in $(-2, 2)$. We do this both for the weighted and unweighted 1-level density (where in the weighted case we use the Petersson weights), thus showing that either formulation may be used. These two 1-level densities differ by a term of size $1/\log (k^2 N)$. Finally, we show that there is another way of extending the sums arising in the Ratios Conjecture, leading to a different answer (although the answer is such a lower order term that it is hopeless to observe which is correct).
\end{abstract}

%\pagebreak

\maketitle

\setcounter{equation}{0}
    %=======================================================
    %   You need the setcounter command above to reset the
    %   equation counting in each section. without it, you'd
    %   have 1.1, 1.2, 1.3, 2.4, 2.5 -- with it (you put it
    %   after each new section) you have 1.1, 1.2, 1.3, 2.1,
    %========================================================

%%%%%%%%%%%%%%%%%%%%%%%%%%%%%%%%%%%%%%%%%%%%%%%%%%%%%%%%%%%%%%%%%%%%%%%%%%%%%%%%%%%%%%%%%%%%%
%%%%%%%%%%%%%%%%%%%%%%%%%%%%%%%%%%%%%%%%%%%%%%%%%%%%%%%%%%%%%%%%%%%%%%%%%%%%%%%%%%%%%%%%%%%%%

\section{Introduction}

Zeros of $L$-functions are some of the most important objects in modern number theory. Numerous problems are connected to them, and frequently the more detailed information we have about zeros, the more we can say about difficult problems. We remark on just a few of these applications. We then discuss a new procedure to predict these properties, and discuss our tests of its predictions.

The Generalized Riemann Hypothesis (GRH) asserts that all non-trivial zeros of an $L$-function have real part $1/2$. Just knowing that there are no zeros on the line $\Re(s) = 1$ for $\zeta(s)$ suffices to prove the Prime Number Theorem. Similarly the non-vanishing of  Dirichlet $L$-functions at $s=1$ imply the infinitude of primes in arithmetic progression (see for example \cite{Da}).

Assuming GRH, all the non-trivial zeros lie on the line $\Re(s)=1/2$. We can thus ask more refined questions about their spacing. The Grand Simplicity Hypothesis asserts that the imaginary parts of zeros of Dirichlet $L$-functions are linearly independent over $\Q$; this is one of the key inputs in Rubinstein and Sarnak's \cite{RubSa} analysis of Chebyshev's bias, the observed preponderance of primes in some arithmetic progressions over others.

Finally, Conrey and Iwaniec \cite{CI} show that if a positive percentage of the spacings between normalized zeros of certain $L$-functions is less than half the average spacing, then the class number of $\Q(\sqrt{-q})$ satisfies $h(q) \gg \sqrt{q} (\log q)^{-A}$ for some $A>0$.

Since the 1970s, random matrix theory has provided powerful models to predict the behavior of zeros of $L$-functions. The scaling limits of zeros of individual or of a family of $L$-functions are well-modeled by the scaling limits of eigenvalues of classical compact groups (see for example \cite{CFKRS,Hej,KaSa1,KaSa2,KeSn1,KeSn2,KeSn3,Mon,Od1,Od2}). In particular, these models immediately imply that a positive percentage of zeros are less than half the average spacing apart.

While the corresponding classical compact group is naturally connected to the monodromy group in the function field case, the connection is far more mysterious for number fields. Further, these models often add the number theoretic pieces in an ad-hoc manner, and thus there is a real need to develop methods which naturally incorporate the arithmetic.\footnote{See \cite{DM2} for some recent results on determining the symmetry group of convolutions of families, and \cite{GHK} for an alternate approach which is a hybrid of the Euler product and the Hadamard expansion, which has the advantage of the arithmetic arising naturally.}

In this work we concentrate on one such approach, the $L$-functions Ratios Conjecture of Conrey, Farmer and Zirnbauer \cite{CFZ1,CFZ2}, which provides a recipe for predicting many properties of $L$-functions to a phenomenal degree, ranging from $n$-level correlations and densities to moments and mollifiers (see \cite{CS} for numerous applications).

In \cite{Mil4} we showed that the Ratios Conjecture successfully predicts all lower order terms up to size $O(N^{-1/2+\gep})$ in the 1-level density for certain families of quadratic Dirichlet characters, at least provided the Fourier transform of the test function is supported in $(-1/3, 1/3)$. In this paper we apply the Ratios Conjecture to families of cuspidal newforms. We chose these families as the 1-level density can be determined for test functions whose Fourier transform is supported in $(-2, 2)$.\footnote{If we assume Hypothesis S from \cite{ILS}, we can extend the number theory calculations up to $(-22/9, 22/9)$; see \eqref{eq:hypothesisS}.} To prove results for support exceeding $(-1, 1)$ requires us to take into account non-diagonal terms, specifically sums of Bessel functions and Kloosterman sums. Thus our hope is that this will be a very good test of the Ratios Conjecture.

\subsection{Notation}

We first set some notation. Let $f \in S_k(N)$, the space of cusp forms of weight $k$ and level $N$, let $\mathcal{B}_k(N)$ be an orthogonal basis of $S_k(N)$, and let $H^\star_k(N)$ be the subset of newforms. To each $f$ we associate an $L$-function:
\begin{equation}
L(s,f)\ =\ \sum_{n=1}^\infty \lambda_f(n) n^{-s}.
\end{equation}
The completed $L$-function is
\begin{equation}\label{eq:completed_L_func}
\Lambda(s,f) \ =\ \left(\frac{\sqrt{N}}{2\pi}\right)^s
\Gamma\left(s+\frac{k-1}{2}\right) L(s,f),
\end{equation}
and satisfies the functional equation $\Lambda(s,f) =
\epsilon_f \Lambda(1-s,f)$ with $\epsilon_f = \pm 1$. Thus
$H^\star_k(N)$ splits into two disjoint subsets, $H^+_k(N) = \{
f\in H^\star_k(N): \epsilon_f = +1\}$ and $H^-_k(N) = \{ f\in
H^\star_k(N): \epsilon_f = -1\}$. We often assume the Generalized Riemann Hypothesis (GRH), namely that all non-trivial zeros of $L(s,f)$ have real part $1/2$.

From Equation $2.73$ of \cite{ILS} we have for $N > 1$ that
\begin{equation}\label{eq:number of terms in hkpm}
|H_k^\pm(N)| \ = \ \frac{k-1}{24}N + O\left( (kN)^{\frac{5}{6}}
\right).
\end{equation} If $N=1$ then $H_k^+(1)=H_k^\ast(1)$ if $k\equiv 0 \bmod 4$ and $H_k^-(1)=H_k^\ast(N)$ if $k \equiv 0 \bmod 4$, where $|H_k^\ast(1)| = \frac{k-1}{12}+O(k^{2/3})$.

We let $D_{1,H_k^\ast(N);R}(\phi)$ denote the weighted 1-level density for the family $H_k^\ast(N)$: \be
D_{1,H_k^\ast(N);R}(\phi) \ = \ \sum_{f\in H_k^\ast(N)} \omega_f^\ast(N)
\sum_{\gamma_f \atop L(1/2+i\gamma_f,f) = 0} \phi\left(\gamma_f
\frac{\log R}{2\pi}\right). \ee We discuss the weights $\omega_f^\ast(N)$ in greater detail in \S\ref{sec:weights}, and $R = k^2 N$ is the analytic conductor, which is constant throughout the family.\footnote{It greatly simplifies our analysis to have a family where the analytic conductors are constant. This allows us to pass the summation over the family past the test function to the Fourier transforms. Non-constant families can often be handled, at a cost of additional work and sieving (see for example \cite{Mil1}).} Katz and Sarnak \cite{KaSa1, KaSa2} conjectured that as the conductors tend to infinity, the 1-level density agrees with the scaling limit of a classical compact group. There are now many cases where, for suitably restricted test functions, we can show agreement between the main terms and the conjectures; see, for example \cite{DM1,FI,Gao,Gu,HR,HM,ILS,KaSa2,Mil1,OS,RR,Ro,Rub,Yo2}. Now that the main terms have been successfully matched in numerous cases, it is natural to try to analyze the lower order terms. Here we break universality. While the arithmetic of the family does not enter into the main terms, it does surface in the lower order term (see for example \cite{FI,Mil2,Mil3,Mil4,Yo1}).

The Ratios Conjecture is a recipe for predicting the main and lower order terms (often up to square-root in the family's cardinality) for ratios of $L$-functions. Consider a family $\FD$ of $L$-functions with some weights $\omega_f$. We shall be particularly interested in both \be R_\FD(\alpha,\gamma) \ = \ \sum_{f \in \mathcal{F}} \omega_f
\frac{L(\foh+\alpha,f)}{L(\foh + \gamma,f)} \ee and $\partial R_\FD(\alpha, \gamma)/\partial \alpha \Big|_{\alpha=\gamma=s}$. We are interested in the derivative as a contour integral of it yields the 1-level density.

\subsection{Weights}\label{sec:weights}

To simplify some of the arguments, we content ourselves with investigating two cases: $k$ is fixed and $N\to\infty$ through the primes\footnote{With additional work, the arguments should generalize to $N$ square-free.}, and $N=1$ and $k \to \infty$.
Throughout our analysis we shall need to investigate sums such as \be \sum_{f \in H_k^\ast(N)} \gl_f(m) \gl_f(n). \ee It is technically easier to consider weighted sums \be \sum_{f \in H_k^\ast(N)} \omega_f(N) \gl_f(m) \gl_f(n), \ee where the $\omega_f(N)$ are the harmonic (or Petersson) weights, though for completeness we study the unweighted sums as well. These are defined by \be \omega_f^\ast(N) \ =\ \frac{\Gamma(k-1)}{(4\pi)^{k-1} (f,f)_N},  \ee where \be (f,f)_N \ = \ \int_{\Gamma_0(N)\setminus \mathbb{H}} f(z) \overline{f}(z) y^{k-2} dxdy. \ee These weights are almost constant. We have the bounds (see \cite{HL,Iw}) \be N^{-1-\gep}\ \ll_k \ \omega_f^\ast(N) \ \ll_k \ N^{-1+\gep}; \ee if we allow ineffective constants we can replace $N^\gep$ with $\log N$ for $N$ large.

The main tool for evaluating these (weighted) sums is the Petersson formula; we state several useful variants in Appendix \ref{sec:PeterssonFormula}.

If $N>1$ we should use modified weights $\omega_f(N) / \omega(N)$, where \be \omega(N)\ =\ \sum_{f\in H_k^\ast(N)} \omega_f(N).\ee The reason is that our family does not include the oldforms. One advantage of restricting to $N$ prime is that the only oldforms in $S_k(N)$ are forms of level $1$. We know there are only $O(k)$ such forms. As each $\omega_f(N) \ll N^{-1+\gep}$, \be \sum_{f\in H_k^\ast(N)} \omega_f(N) \ = \ \sum_{f\in S_k(N)} \omega_f(N) + O\left(\frac{k}{N^{1-\gep}}\right) \ = \ 1 + O\left(\frac{k}{N^{1-\gep}}\right). \ee Thus for $k$ fixed and $N\to\infty$, the difference between using $\omega_f(N)$ and $\omega_f(N)/\omega(N)$ is $O(N^{-1+\gep})$. We set \be\label{eq:omegaastfN} \twocase{\omega_f^\ast(N) \ = \ }{\omega_f(1)}{if $N=1$}{\omega_f(N)/\omega(N)}{if $N>1$;} \ee note \be\label{eq:omegaastfNb} \sum_{f \in H_k^\ast(N)} \omega_f^\ast(N) \ = \ 1 \ = \ \left(1 + O\left(N^{-1+\gep}\right)\right) \sum_{f \in \mathcal{B}_k(N)} \omega_f(N). \ee

\begin{rek} For some problems, such as bounding the order of vanishing at the central point for families of cuspidal newforms \cite{HM, ILS}, it is desirable to study the unweighted family. We shall see below that there is a difference of size $1/\log R$ between the weighted and unweighted 1-level densities. The predictions from the Ratios Conjecture (for weighted and unweighted families) agrees with number theory in both cases.\end{rek}

\subsection{Main results}

\begin{thm}\label{thm:ratiosconj1level1}
Assume GRH for $\zeta(s)$ and all
$L(s,f)$ with $f\in H_k^\ast(N)$. The Ratios Conjecture predicts
\bea\label{eq:d1hkastnrphitemp2} & & D_{1,H_k^\ast(N);R}(\phi) \ = \
2\sum_p \hphi\left(\frac{2 \log p}{\log R}\right) \frac{\log
p}{p\log R} + M(\phi) \nonumber\\
& &  \ \ \ \ \ +\ \frac1{\log R} \int_{-\infty}^\infty \left(2 \log
\frac{\sqrt{N}}{\pi} + \psi\left(\frac14 + \frac{k\pm 1}4
+\frac{2\pi i t}{\log R}\right) \right) \phi(t) dt \nonumber\\ & & \ \ \ \ \ + \ O\left((kN)^{-1/2+\gep}\right), \eea
where \bea M(\phi) & \ = \ &  \frac{2i^k \mu(N)}{N\log
R} \int_{-\infty}^\infty X_L\left(\foh + \frac{2\pi i t}{\log
R}\right) \frac{\zeta(2)}{\zeta(2+\frac{8\pi i t}{\log
R})}\ \zeta\left(1+\frac{4\pi i t}{\log R}\right) \nonumber\\ & & \ \
\ \ \ \cdot \ \prod_p \left(1 - \frac{p^{4\pi i t/\log
R}-1}{p(p^{1+4\pi i t/\log R}+1)}\right) e^{-2\pi i t \frac{\log N}{\log R}} \phi(t) dt \eea and the factor $\frac{\mu(N)}{N}\exp\left(-2\pi i t \log N/ \log R\right)$ is not present if $N=1$. If $N > 1$ then $M(\phi) \ll N^{-1}$. Let $\supp(\hphi) \subset (-\sigma, \sigma)$. If $N=1$ then $M(\phi) \ll 2009^k k^{-\frac{1-4\sigma}3k}$, which decays more rapidly than $k^{-\delta}$ for any $\delta > 0$ provided $\sigma < 1/4$.
\end{thm}

\begin{rek} Our estimate for $M(\phi)$ is significantly worse when $N=1$; see Remark \ref{rek:whyn=1harderntoinfinity} for an explanation and a connection to other problems. Interestingly, if we change the order of some of the steps in the Ratios Conjecture's recipe, then $M(\phi)$ changes by a factor of $e^{-\gamma}$. See Appendix \ref{sec:mertenssumextension} for complete details, as well as Remark \ref{rek:egammaMertens}.
\end{rek}

The 1-level density computation has some differences depending on whether or not $N\to\infty$ through the primes or $N=1$ and $k\to\infty$. We therefore separate our results into two cases. Further, we can often obtain results for smaller support without assuming GRH for Dirichlet $L$-functions, and thus we isolate these as well.

\begin{thm}\label{thm:1ldnumbthN} Let $\supp(\hphi) \subset (-\sigma, \sigma)$ and let $N\to\infty$ through the primes.

\bi

\item (\emph{Density Theorem Limited}) If $\sigma < 3/2$ then the weighted 1-level density for the family $H_k^\ast(N)$ agrees with the prediction from the Ratios Conjecture up to errors of size $O(N^{\sigma-\frac32+\gep}$ $+$ $N^{\frac{\sigma}2-1+\gep}$ $+$ $N^{\frac{\sigma}{4}-1+\gep''})$.

\item (\emph{Density Theorem Extended}) Assuming GRH for $\zeta(s)$, all Dirichlet $L$-functions and $L(s,f)$, the weighted 1-level density agrees with the prediction  the Ratios Conjecture up to errors of size $O(N^{\frac{\sigma}2-1+\gep} + N^{\frac{\sigma}{4}-1+\gep''})$.

\ei
\end{thm}

\begin{rek} Theorem \ref{thm:1ldnumbthN} implies we have agreement up to a power savings in $N$ for $\sigma < 3/2$, and up to square-root cancelation for $\sigma < 1$. Assuming GRH, we can extend agreement up to $\sigma < 2$, again saving a power in $N$.\end{rek}

\begin{thm}\label{thm:1ldnumbthk}  Let $\supp(\hphi) \subset (-\sigma, \sigma)$ $N=1$ and $k, K\to\infty$.

\bi

\item (\emph{Density Theorem Limited}) The weighted 1-level density for the family $H_k^\ast(1)$ agrees with the prediction from the Ratios Conjecture up to errors of size $O(k^{-(5-3\sigma)/6+\gep})$ for $\sigma < 1/4$. If we knew $M(\phi) \ll k^{-(5-3\sigma)/6+\gep}$ for $\sigma < 1$, then we would have agreement up to $\sigma < 1$.

\item (\emph{Density Theorem Extended}) Let $h$ be a Schwartz function compactly supported on $(0,\infty)$. Consider a weighted average (over $k$) of the weighted $1$-level density \bea \mathcal{A}^\ast(K;\phi) & \ = \ & \frac1{A^\ast(K)} \sum_{k\equiv 0 \bmod 2} \frac{24}{k-1} h\left(\frac{k-1}{K}\right) \sum_{f\in H_k^\ast(1)} D_{1,H_k^\ast(1);k^2}(\phi), \eea where \bea A^\ast(K) \ = \ \sum_{k\equiv 0 \bmod 2} \frac{24}{k-1} h\left(\frac{k-1}{K}\right) \left|H_k^\ast(1)\right| \ = \ \widehat{h}(0) K + O(K^{2/3}). \eea Assuming GRH for $\zeta(s)$, all Dirichlet $L$-functions and $L(s,f)$, the 1-level density agrees with the prediction the Ratios Conjecture up to errors of size $O(K^{-(5-\sigma)/6+\gep}$ $+$ $K^{\sigma-2+\gep})$ for $\sigma < 1/4$. If we knew $M(\phi) \ll K^{-(5-\sigma)/6+\gep}$ $+$ $K^{\sigma-2+\gep}$ for $\sigma < 2$, then we would have agreement up to $\sigma < 2$.
    
    \item (Hypothesis S and Density Theorem Extended) Assume Hypothesis S from \cite{ILS} (i.e., \eqref{eq:hypothesisS}) with $A=0$ and $\alpha = 1/2$. Then as $K\to\infty$ the weighted average (over $k$) of the weighted 1-level density agrees with the prediction from the Ratios Conjecture for $\sigma < 1/4$. If we knew $M(\phi)$ $\ll$ $K^{-2(2.5-\sigma)}$ $+$ $K^{-(5-\sigma)/6+\gep}$ $+$ $K^{-\frac{11}{2}(1-\frac{9}{22}\sigma)}$ for $\sigma < 22/9$, then we would have agreement up to $\sigma < 22/9$.
\ei

\end{thm}

\begin{rek} Theorem \ref{thm:1ldnumbthk} implies we have agreement up to a power savings in $K$ for $\sigma < 1/4$; in fact, we agree \emph{beyond} square-root cancelation in this range. Assuming GRH, by averaging over $k$  we can extend our calculations up to $\sigma < 2$ (or, if we assume Hypothesis S, up to $\sigma < 22/9$). If we knew $M(\phi)$ were small, we would again save a power in $N$ (with agreement up to square-root cancelation for $\sigma < 3/2$ if we assume GRH for Dirichlet $L$-functions, or up to $\sigma < 20/9$ if we assume Hypothesis S). \end{rek}

\begin{thm}\label{thm:1ldnumbthunweighted} Assume GRH for $\zeta(s)$, all Dirichlet $L$-functions and all $L(s,f)$. The unweighted 1-level density for $H_k^\ast(N)$ agrees with the predictions of the Ratios Conjecture for the unweighted family, up to a power savings in the family's cardinality, as $N\to\infty$ through the primes; this answer differs from the weighted 1-level density by an additional term of size $1/\log R$.
The Ratios Conjecture applied to the unweighted family predicts
\bea\label{eq:d1hkastnrphitemp2unwt} D_{1,H_k^\ast(N);R}^{{\rm unwt}}(\phi) & \ = \ & \frac1{\log R}
\int_{-\infty}^\infty \left(2 \log \frac{\sqrt{N}}{\pi} +
\psi\left(\frac14 + \frac{k\pm 1}4 +\frac{2\pi i t}{\log R}\right)
\right) \phi(t) dt \nonumber\\ & & \ \ \  +2 \sum_{\nu \equiv 0 \bmod 2 \atop \nu \ge 2} \sum_{p \neq N} \frac{p-1}{p^\nu} \hphi\left(\nu \frac{\log p}{\log R}\right)\frac{\log p}{\log R}\nonumber\\ & &  \ \ \ \ \ + \ O\left((kN)^{-1/2+\gep}\right), \eea which agrees with number theory up to errors of size $O(N^{-(2-\sigma)/6+\gep})$.
\end{thm}

\begin{rek} Theorem \ref{thm:1ldnumbthunweighted} implies that the predictions from the $L$-functions Ratios Conjecture agree with number theory for both the weighted and unweighted families. Thus,
when investigating cuspidal newforms, we may study either family.  \end{rek}

The paper is organized as follows. In \S\ref{sec:ratiosconjintro} we describe the Ratios Conjecture's recipe, and determine its prediction for the 1-level density for our families. In \S\ref{sec:1ldfromratiosconjexpansiontheorems} we analyze these predictions and prove Theorem \ref{thm:ratiosconj1level1}. In \S\ref{sec:onelevelfromnt} we prove Theorems \ref{thm:1ldnumbthN} and \ref{thm:1ldnumbthk}, which show the 1-level densities agree (up to a power savings in the cardinality of the families, at least for suitably restricted test functions) with what can be proved. Finally, in \S\ref{sec:extendingsuppunweightedS1} we analyze the unweighted 1-level density, and prove Theorem \ref{thm:1ldnumbthunweighted}.

%%%%%%%%%%%%%%%%%%%%%%%%%%%%%%%%%%%%%%%%%%%%%%%%%%%%%%%%%%%%%%%%%%%%%%%%%%%%%%%%%%%%%%%%%%%%%
%%%%%%%%%%%%%%%%%%%%%%%%%%%%%%%%%%%%%%%%%%%%%%%%%%%%%%%%%%%%%%%%%%%%%%%%%%%%%%%%%%%%%%%%%%%%%
%\setcounter{equation}{0}

\section{Ratios Conjecture}\label{sec:ratiosconjintro}

The Ratios Conjecture is a recipe to predict the main and lower order terms for a variety of problems. We analyze its predictions for the 1-level density for families of cuspidal newforms. We first briefly describe its recipe for predicting quantities related to \be R_\FD(\alpha,\gamma) \ = \ \sum_{f \in \mathcal{F}} \omega_f
\frac{L(\foh+\alpha,f)}{L(\foh + \gamma,f)}. \ee

\ben

\item Use the approximate functional equation to expand the numerator into two sums plus a remainder. The first sum is over $m$ up to $x$ and the second over $n$ up to $y$, where $xy$ is of the same size as the analytic conductor (typically one takes $x=y$). We ignore the remainder term.

\item Expand the denominator by using the generalized Mobius function: \be \frac1{L(s,f)} \ = \ \sum_h \frac{\mu_f(h)}{h^s}, \nonumber\ \ee where $\mu_f(h)$ is the multiplicative function equaling 1 for $h=1$, $-\lambda_f(p)$ if $h=p$, $\chi_0(p)$ if $h=p^2$ (with $\chi_0$ the trivial character modulo $N$) and $0$ otherwise.

\item Execute the sum over $\mathcal{F}$, keeping only main (diagonal) terms.

\item Extend the $m$ and $n$ sums to infinity (i.e., complete the products).

\item Differentiate with respect to the parameters, and note that the size of the error term does not significantly change upon differentiating.

\item A contour integral involving $\frac{\partial}{\partial \alpha}R_\FD(\alpha,\gamma)\Big|_{\alpha=\gamma=s}$ yields the 1-level density.

\een

We now describe these steps in greater detail and deduce the Ratios Conjecture's prediction for the 1-level density.

\begin{rek} It is almost miraculous how well the Ratios Conjecture works, given that several of the steps involve throwing away significant error terms. The miracle is that all these errors seem to cancel, and the resulting expression is correct to a remarkable order. See Remark \ref{rek:commenterrorsextendingprod} for more details. \end{rek}

\begin{rek}\label{rek:diffexpranalyerror} Differentiating is essentially harmless because we have analytic functions. If the error were $N^{-1/2} \cos(N^2 \alpha)$  and $\alpha$ was forced to be real, then differentiating increases the error from size $N^{-1/2}$ to $N^{3/2}$! For us, $\alpha$ will be complex. By Cauchy's integral theorem, if $f$ is analytic at $z_0$ then \be f'(z_0) \ = \ \frac1{2\pi i} \oint_C \frac{f(z)}{(z-z_0)^2} dz, \ee where $C$ is a circle of very small radius about $z_0$. The sum of the ratios is analytic, and we shall see later that the main term is analytic. Thus their difference, the error term, is also analytic. Applying Cauchy's argument with a circle of very small radius, say $\log^{-2009} R$, we see the effect of differentiating is only to increase the error by some powers of $\log R$. We thank David Farmer for pointing this out to us.
\end{rek}

\subsection{Approximate Functional Equation}

We state the approximate functional equation in greater generality than we need, though not the greatest
generality possible; see Section 1 of \cite{CFKRS} for more details.
Let \be L(s)\ =\ \sum_{n=1}^\infty \frac{a_n}{n^s} \ee be a nice
$L$-function with real coefficients ($a_n \in \R$), \be \gamma_L(s)
\ = \ P(s) Q^s \prod_{j=1}^w \Gamma(w_j s + \mu_j) \ee with $Q, w_j
> 0$, $\mu_j \ge 0$ and $P(s)$ a real polynomial whose zeros in
$\Re(s) > 0$ are at the poles of $L(s)$ (so if $L(s)$ has no poles
then $P(s)$ is constant). Let \be \xi_L(s) \ = \ \gamma_L(s) L(s) \
= \ \epsilon \overline{\xi}_L(1-s) \ee be the completed
$L$-function, with $|\epsilon| = 1$ the sign of the functional
equation and $\overline{\xi}_L(s) = \overline{\xi_L(\overline{s})}$.
Our assumptions imply that $\overline{\xi}_L(s) = \xi_L(s)$. Set \be
X_L(s) \ = \ \frac{\gamma_L(1-s)}{\gamma_L(s)} \ = \
\frac{P(1-s)Q^{1-s} \prod_{j=1}^w \Gamma(w_j(1-s)+\mu_j)}{P(s)Q^s
\prod_{j=1}^w \Gamma(w_js + \mu_j)}. \ee Then

\begin{lem}[The Approximate Functional Equation]\label{lem:approxfnaleq}
Notation and assumptions as above, \be L(s) \ = \ \sum_{m \le x}
\frac{a_m}{m^s} + \epsilon X_L(s) \sum_{n \le y}
\frac{a_n}{n^{1-s}} + {\rm remainder}, \ee where $xy$ is of the same size as the analytic conductor.
\end{lem}

\begin{rek} The Ratios Conjecture's recipe for generating predictions ignores the remainder term in the approximate functional equation. Thus we too shall ignore these errors in our arguments below, and treat the approximate functional equation as exact.
\end{rek}

For us, $L(s)$ will be a weight $k$ cuspidal newform of level $N$,
which we shall denote by $L(s,f)$. In this case, we have (see
\cite{ILS} for instance) that \bea \gamma_L(s) & \ = \ &
\left(\frac{2^k}{8\pi}\right)^{1/2}
\left(\frac{\sqrt{N}}{\pi}\right)^s \Gamma\left(\frac{s}2 +
\frac{k-1}4\right) \Gamma\left(\frac{s}2+\frac{k+1}4\right)
\nonumber\\ &=& \left(\frac{\sqrt{N}}{2\pi}\right)^s
\Gamma\left(s+\frac{k-1}2\right); \eea note that $\gamma_L(s)$
depends only on the weight $k$ and the level $N$ of the cuspidal
newform $f$. This yields the following expressions for $X_L(s)$:
\bea\label{eq:expansionsXLs} X_L(s)& \ = \ &
\left(\frac{\sqrt{N}}{\pi}\right)^{1-2s}
\frac{\Gamma\left(\frac{1-s}2+\frac{k-1}4\right)
\Gamma\left(\frac{1-s}2+\frac{k+1}4\right)}{
\Gamma\left(\frac{s}2+\frac{k-1}4\right)
\Gamma\left(\frac{s}2+\frac{k+1}4\right)}\nonumber\\ & \ = \ &
\left(\frac{\sqrt{N}}{2\pi}\right)^{1-2s}
\frac{\Gamma\left(\frac{1-s}2+\frac{k-1}2\right)}{
\Gamma\left(\frac{s}2+\frac{k-1}2\right)}. \eea

Finally, the analytic conductor of a cuspidal newform of weight $k$ and level $N$ is (up to a constant) $k^2N$. Thus we will typically take $x=y\sim \sqrt{k^2 N}$ in the approximate functional equation.

\subsection{Ratios Conjecture}

Let $\chi_0$ denote the principal character with conductor $N$. For
$f$ a weight $k$ cuspidal newform of level $N$ we have
\bea\label{eq:LsfInvLsf} L(s,f) & \ = \ & \prod_p \left(1 -
\frac{\lambda_f(p)}{p^s} + \frac{\chi_0(p)}{p^{2s}}\right)^{-1} \ =
\ \sum_{n=1}^\infty \frac{\lambda_f(n)}{n^s} \nonumber\\
\frac1{L(s,f)} &\ = \ & \prod_p \left(1 - \frac{\lambda_f(p)}{p^s} +
\frac{\chi_0(p)}{p^{2s}}\right) \ = \ \sum_{n=1}^\infty
\frac{\mu_f(n)}{n^s}, \eea where $\mu_f(n)$ is the multiplicative
function such that $\mu_f(1) = 1$, $\mu_f(p) = -\lambda_f(p)$,
$\mu_f(p^2) = \chi_0(p)$, and $\mu_f(p^k) = 0$ for $k \ge 3$.

Let $\mathcal{F}$ be a family of weight $k$ cuspidal newforms of
level $N$. The Ratios Conjecture for the family gives an expansion
for \be R_\FD(\alpha,\gamma) \ = \ \sum_{f \in \mathcal{F}} \omega_f
\frac{L(\foh+\alpha,f)}{L(\foh + \gamma,f)}, \ee where
$\alpha$ and $\gamma$ satisfy \ben \item $\Re(\alpha) \in (-1/4, 1/4)$; \item $\Re(\gamma) \in (1/\log |\FD|, 1/4)$; \item $\Im(\alpha), \Im(\gamma) \ll_\gep |\FD|^{1-\gep}$ for all $\gep > 0$. \een

We have introduced weights $\omega_f$, as often in practice the
weighted sum is significantly easier to control. For example, we may
take $\omega_f$ to be the Petersson weights, which facilitates
applying the Petersson formula (see Appendix \ref{sec:PeterssonFormula} for statements). As remarked in \S\ref{sec:weights}, it is convenient to choose $\omega_f = \omega_f^\ast(N)$ (see \eqref{eq:omegaastfN}).

We shall concentrate on the diagonal terms in the Petersson formula.
Thus if our family is $H_k^\ast(N)$ then by the Petersson formula we
have \emph{for $n_1$ and $n_2$ relatively prime to $N$}, \be \sum_{f
\in H_k^\ast(N)} \omega_f^\ast(N) \lambda_f(n_1) \lambda_f(n_2) \ = \
\delta_{n_1,n_2} + {\rm small}. \ee The weights are  normalized to sum to 1. If $N>1$ our sums do not include the oldforms; however, the oldforms do not contribute to the main term of the Petersson formula in this case.

In general, we must
be careful by what we mean by `small' when we apply the Petersson
formula. The error term is a Bessel-Kloosterman sum, and is
typically small only if $n_1$ and $n_2$ are not too large with
respect to $k$ and $N$ (and are relatively prime to $N$). It is
\emph{very} important that our sums are restricted. It is only after
we compute the main term that the heuristics of the Ratios
Conjecture tells us to extend the sums to infinity. Depending on how
(and when!) we extend our sums to infinity can lead to different
answers. \footnote{These differences, however, involve terms of size $1/N$, which is unimaginable beyond anything we can hope to prove. Interestingly, however, the difference between these two terms is related to sieving actual versus random primes. See Appendix \ref{sec:mertenssumextension}.}

Unless our family is all of $H_k^\ast(N)$ and $N=1$, however, the
sign of the functional equation is not constant. For square-free $N$
we have \be\label{eq:signfneq} \epsilon_f \ = \ i^k \mu(N)
\lambda_f(N) \sqrt{N}; \ee thus the sign of the functional equation
only weakly depends on the specific form $f$. Further
$\lambda_f(q)^2 = 1/q$ if $q|N$. Note $\mu(1) = 1$ and $\mu(N) = -1$
if $N$ is prime, and since $k$ is even we have $i^k = \pm 1$. Thus there is at most one `bad' prime, namely $N$.

\begin{rek} We consider just the case
$\mathcal{F} = H_k^\ast(N)$ with $N$ either $1$ or prime here; more
involved arguments should be able to handle the case of $N$
square-free, and we will investigate the sub-families $H_k^\pm(N)$ in a future paper. \end{rek}

\begin{lem}\label{lem:ratiosconjRalphagamma}
The Ratios Conjecture predicts that \bea R_{H_k^\ast(N)}(\alpha,\gamma)& \ = \  & \prod_{p} \left(1 -
\frac{1}{p^{1+\alpha+\gamma}} + \frac1{p^{1+2\gamma}}\right)
 \nonumber\\ & & \ -\
\frac{i^k\mu(N)
X_L\left(\foh+\alpha\right)}{N^{1+\gamma}\zeta(1-\alpha+\gamma)} \prod_{p} \left(1+\frac{p^{1-\alpha+\gamma}}{p^{1+2\gamma}(p^{1-\alpha+\gamma}-1)}\right) \nonumber\\ & & \ + \ O\left(\left|H_k^\ast(N)\right|^{-1/2+\gep}\right),
\eea where the $N$-factors are present only if $N$ is prime.
\end{lem}

\begin{proof}
From the Approximate Functional Equation (Lemma
\ref{lem:approxfnaleq}) and \eqref{eq:signfneq} we have \be
L\left(\foh+\alpha,f\right) \ = \ \sum_{m \le x}
\frac{\lambda_f(m)}{m^{\foh+\alpha}} + \ i^k \mu(N)\lambda_f(N)
\sqrt{N} X_L\left(\foh+\alpha\right) \sum_{n\le y}
\frac{\lambda_f(n)}{n^{\foh-\alpha}}, \ee where $x = y \sim \sqrt{k^2 N}$. From \eqref{eq:LsfInvLsf}
we have \be \frac1{L(s,f)} \ = \ \sum_{h=1}^\infty
\frac{\mu_f(h)}{h^{\foh+\gamma}}. \ee Therefore
\bea\label{eq:Ralphagammatwopieces} & & R_{H_k^\ast(N)}(\alpha,\gamma) \ = \
\nonumber\\ & & \sum_{f\in H_k^\ast(N)} \omega_f^\ast(N) \left[ \sum_{m \le
x \atop h} \frac{\mu_f(h)\lambda_f(m)}{h^{\foh+\gamma}
m^{\foh+\alpha}} + i^k\mu(N)X_L\left(\foh+\alpha\right) \sqrt{N}
\sum_{n \le y \atop h} \frac{\mu_f(h) \lambda_f(N)
\lambda_f(n)}{h^{\foh+\gamma} m^{\foh-\alpha}} \right].\nonumber\\
\eea

If $N > 1$ then the presence of the $\lambda_f(N)$ factor requires
us to handle the two sums in slightly different manners. We first analyze the sum without the $\lambda_f(N)$ factor. By the Petersson formula, we have $\sum_{f \in H_k^\ast(N)} \omega_f^\ast(N)
\lambda_f(n_1)\lambda_f(n_2) = \delta_{n_1,n_2}+ {\rm small}$ if at least one of $n_1$ and $n_2$ is relatively prime to $N$. There are two cases: either $N=1$ and $k\to\infty$ or $k$ is fixed and $N\to\infty$ through the primes. As $x \sim \sqrt{k^2 N}$, if $N>1$ then $N$ does not divide $m$ for sufficiently large $N$. Thus we may assume $(n_2,N)=1$. Using the
multiplicativity of the Fourier coefficients, from the Petersson formula (Lemma \ref{lem:ils23}) we see that if $p|n_1$
then there is negligible contribution unless $p^\ell||n_1$ and $p^\ell||n_2$. From the definition of the multiplicative function $\mu_f(h)$, we see immediately that $h$ must be cube-free (if not, $\mu_f(h) = 0$). Thus we may write $h = p_1 \cdots p_r \cdot q_1^2 \cdots q_\ell^2$ where $p_1,\dots, q_\ell$ are distinct primes, and $\mu_f(h) = (-1)^r \lambda_f(p_1\cdots p_r) \chi_0(q_1\cdots q_\ell)$. We immediately see that unless $m$ is square-free and equal to $p_1 \cdots p_r$ and the $q_i$ are relatively prime to $N$ then the main term from $\mu_f(h)\lambda_f(m)$ is zero. Further, the $p_i$ must also be prime to $N$, as $p_i \le m \le x \sim \sqrt{k^2 N}$.
Thus the only contribution from the $m$ and $h$-sum is \be \prod_{p\le x} \left(1 -
\frac{\lambda_f(p)^2}{p^{1+\alpha+\gamma}} +
\frac1{p^{1+2\gamma}}\right) \cdot \prod_{p > x \atop p \neq N} \left(1+\frac{1}{p^{1+2\gamma}}\right).\ee
To see this, use multiplicativity to replace the sum in \eqref{eq:Ralphagammatwopieces} with a product over primes, dropping all terms which will give a negligible contribution after applying the Petersson formula. For
each prime $p \le x$ we either have $1$, $\mu_f(p)\lambda_f(p)$ or
$\mu_f(p^2) \lambda_f(1)$. The product over $p > x$ arises from the
fact that, for such large primes, we must either have $1$ or
$\mu_f(p^2)\lambda_f(1)$ (as the $m$-sum is only up to primes at
most $x$, and the prime $p=N$ can be ignored because $\chi_0(N)=0$). Thus when
we use the Petersson formula we always have two Fourier
coefficients relatively prime to the level $N$. Summing over $f\in
H_k^\ast(N)$ allows us to replace $\lambda_f(p)^2$ with $1 +
{\rm small}$ (and, as always, we ignore all `small' terms), so the first half of $R_{H_k^\ast(N)}(\alpha,\gamma)$ is \be
\prod_{p \le x} \left(1 -
\frac{1}{p^{1+\alpha+\gamma}} + \frac1{p^{1+2\gamma}}\right) \cdot \prod_{p > x \atop p \neq N} \left(1+\frac{1}{p^{1+2\gamma}}\right). \ee As is customary
in applications of the Ratios Conjecture, we complete the $m$-sum by
extending it to infinity. This is equivalent to sending $x$ to
infinity. Thus the first term of $R_{H_k^\ast(N)}(\alpha,\gamma)$ is \be \prod_{p} \left(1 - \frac{1}{p^{1+\alpha+\gamma}} +
\frac1{p^{1+2\gamma}}\right).\ee

We now study the $\lambda_f(N)\mu_f(h)\lambda_f(n)$ terms in
\eqref{eq:Ralphagammatwopieces}, noting that $N$ does not divide $n$ (since $n \le y \sim \sqrt{k^2 N}$). There is thus negligible contribution unless
$N||h$. For $p=N$ the factor is now \be
\frac{\mu_f(N)}{N^{\foh+\gamma}} \ = \
-\frac{\lambda_f(N)}{N^{\foh+\gamma}}
\ee (again, this factor is not present if $N=1$). Remember we have a
truncated sum, with $n \le y$. Thus for $p \le y$ the
factors are the same as before (except we replace $\alpha$ with
$-\alpha$), arising from factors of $1$, $\mu_f(p)\lambda_f(p)$ or
$\mu_f(p^2)\lambda_f(1)$. However, for $y < p
\neq N$ the factor is $1 + p^{-1-2\gamma}$ (arising from $1$ or
$\mu_f(p^2)$ -- there is no $\mu_f(p)\lambda_f(p)$ term as $p > y$).
Thus our factors are \bea & &
i^k\mu(N)\lambda_f(N)\sqrt{N}X_L\left(\foh + \alpha\right)
\nonumber\\ & & \ \ \ \cdot \ \prod_{p\le y } \left(1
- \frac{\lambda_f(p)^2}{p^{1-\alpha+\gamma}} +
\frac1{p^{1+2\gamma}}\right) \cdot \prod_{p > y \atop p \neq N}
\left(1+\frac1{p^{1+2\gamma}}\right) \cdot
\frac{-\lambda_f(N)}{N^{\foh+\gamma}},\ \ \ \ \ \ \ \ \ \ \ \ \eea where as before the $N$-factor
is present only if $N$ is prime. If $N>1$ we replace $\lambda_f(N)^2$ with
$1/N$, so when we apply the Petersson formula all Fourier
coefficients will be relatively prime to the level $N$. If $N=1$ we do not have this second factor of $\lambda_f(N)$; however, as $\lambda_f(1) = 1$ the resulting expression is the same.

Summing over
$f\in H_k^\ast(N)$ allows us to replace the
$\lambda_f(p)^2$ factors above with $1 + {\rm small}$. Thus the
product becomes \bea & & -\frac{i^k \mu(N)}{N^{1+\gamma}} X_L\left(\foh +
\alpha\right)\prod_{p\le y} \left(1 -
\frac{1}{p^{1-\alpha+\gamma}} + \frac1{p^{1+2\gamma}}\right) \cdot
\prod_{p > y} \left(1+\frac1{p^{1+2\gamma}}\right) \cdot
\nonumber\\ & = \ & -\frac{i^k \mu(N)}{N^{1+\gamma}} X_L\left(\foh + \alpha\right)
\nonumber\\ & & \ \ \ \cdot\ \prod_{p \le y}
\left(1-\frac1{p^{1-\alpha+\gamma}}\right) \cdot
\left(1+\frac{p^{1-\alpha+\gamma}}{p^{1+2\gamma}(p^{1-\alpha+\gamma}-1)}\right)
\cdot \prod_{p > y} \left(1+\frac1{p^{1+2\gamma}}\right). \nonumber\\
\eea As before, we complete the $n$-sum by sending $y$ to infinity.
We have deliberately pulled out the $p$-factor of
$1/\zeta(1-\alpha+\gamma)$ to improve the convergence of the
remaining piece. We thus find this factor is \bea & & -\frac{i^k \mu(N)}{N^{1+\gamma}} X_L\left(\foh + \alpha\right)
\cdot \frac{1}{\zeta(1-\alpha+\gamma)}
\prod_{p}
\left(1+\frac{p^{1-\alpha+\gamma}}{p^{1+2\gamma}(p^{1-\alpha+\gamma}-1)}\right).\ \ \eea
Substituting the above completes the proof.
\end{proof}

\begin{rek} In the Ratios Conjecture, the size of the error term is added in a somewhat ad-hoc manner. The predicted size of the error term is amazing, as it implies the lower order terms depending on the arithmetic of the family are calculated basically up to square-root cancelation in the family's cardinality. As the recipe involves throwing away numerous remainders and arguing their aggregate does not matter, it is not possible to rigorously derive the size of the error term (unless, of course, we make significant progress towards proving the Ratios Conjecture!), and the standard assumptions in practice are that it is typically smaller than the main term by approximately the square-root of the family's cardinality. See Remark \ref{rek:commenterrorsextendingprod} for additional comments on the discarded error terms. \end{rek}

\begin{lem}\label{lem:Rprimealphagammar}
Let $R_{H_k^\ast(N)}'(r,r) = \frac{d}{d\alpha}
R_{H_k^\ast(N)}(\alpha,\gamma)\Big|_{\alpha=\gamma=r}$. Then for $\Re{r} > 0$ the Ratios Conjecture predicts \bea R_{H_k^\ast(N)}'(r,r) & \ = \ & \sum_{p} \frac{\log p}{p^{1+2r}}
+\frac{i^k\mu(N)}{N^{1+r}} X_L\left(\foh+r\right) \prod_{p}
\left(1+\frac{1}{(p-1)p^{2r}}\right)\nonumber\\ & & \ \ + \ O\left(\left|H_k^\ast(N)\right|^{-1/2+\gep}\right),
\eea where, as always, the $N$-factors are present only if $N>1$.
\end{lem}

\begin{proof} We must differentiate the two terms in Lemma
\ref{lem:ratiosconjRalphagamma}, and investigate the limit as
$y\to\infty$; see Lemma \ref{rek:diffexpranalyerror} for an explanation as to why the size of the error term is unaffected. The first term is easily handled. Using $d \log
f(\alpha) / d\alpha = f'(\alpha)/f(\alpha)$, we see that \bea
\frac{d}{d\alpha} \left[ \prod_{p} \left(1 -
\frac{1}{p^{1+\alpha+\gamma}} + \frac1{p^{1+2\gamma}}\right) \right]\Bigg|_{\alpha=\gamma=r}
& = & \prod_{p} \left(1 - \frac{1}{p^{1+\alpha+\gamma}} +
\frac1{p^{1+2\gamma}}\right) \Bigg|_{\alpha=\gamma=r}
\nonumber\\ & &  \cdot\ \frac{d}{d\alpha}
\log\left[\prod_{p}\left(1 - \frac{1}{p^{1+\alpha+\gamma}} +
\frac1{p^{1+2\gamma}}\right)
\right]\Bigg|_{\alpha=\gamma=r}\nonumber\\ & = & \sum_{p} \frac{\log p}{p^{1+2r}}, \eea where we need
$\Re(r)
> 0 $ to ensure that the sum converges.

We now handle the second term in Lemma
\ref{lem:ratiosconjRalphagamma}. We must differentiate, with respect
to $\alpha$, \be\label{eq:lemRprimealphagammar2ndterm}
-\frac{i^k\mu(N)
X_L\left(\foh+\alpha\right)}{N^{1+\gamma}\zeta(1-\alpha+\gamma)} \prod_{p}
\left(1+\frac{p^{1-\alpha+\gamma}}{p^{1+2\gamma}(p^{1-\alpha+\gamma}-1)}\right).\ee We use the
following observation (see page 7 of \cite{CS}): if $f(z,w)$ is
analytic at $(z,w) = (r,r)$, then \be \frac{d}{d\alpha}
\frac{f(\alpha,\gamma)}{\zeta(1-\alpha+\gamma)}\Bigg|_{\alpha=\gamma=r}
\ = \ -f(r,r). \ee Thus the derivative of
\eqref{eq:lemRprimealphagammar2ndterm} with respect to $\alpha$,
evaluated at $\alpha = \gamma = r$, is \be \frac{i^k\mu(N)}{N^{1+r}}
X_L\left(\foh+r\right) \prod_{p}
\left(1+\frac{1}{(p-1)p^{2r}}\right).\ee
\end{proof}

\begin{rek}\label{rek:egammaMertens}
If we don't extend the sums to infinity before differentiating, we get from Mertens' theorem (see Appendix \ref{sec:mertenssumextension}) a factor of $e^{-\gamma}$ in the second sum, where $\gamma$ here is Euler's constant. This is very interesting, as $e^{-\gamma}$ is related to sieving primes. The sieving constant of $e^{-\gamma}$ in Mertens theorem is not 1, though for a generic sequence of random primes (also called Hawkins primes) it is. While it is fascinating that there are two procedures which lead to different answers, this term is of size $1/N$, well beyond any plausible hope of testing.\footnote{This term is related to $M(\phi)$. If $N=1$ we can only show $M(\phi)$ is small for $\sigma < 1/4$, though based on number theory computations we expect it to be small for $\sigma < 2$ or even $22/9$.} See \cite{BK,Ha,HW,Gr,NW,Wu} for some additional comments on $e^{-\gamma}$.
\end{rek}

\begin{rek}\label{rek:commenterrorsextendingprod} We briefly comment on the size of the errors made at various steps in the Ratios Conjecture. For example, consider the first piece of $R_{H_k^\ast(N)}'(r,r)$, namely $\sum_{p} \frac{\log p}{p^{1+2r}}$. This piece arose from a product originally over $p \le \sqrt{R}$ which we extended to be over all $p$; thus the error between what we should have had and what we wrote is $\sum_{p \ge \sqrt{R}} \frac{\log p}{p^{1+2r}}$. We typically evaluate this when $r = \gep + it$, and thus we have introduced an error of size $O(R^{-\gep'})$. Thus while this is smaller than any power of $1/\log R$, it is significantly more than $R^{-1/2+\gep}$. Thus this sizable error \emph{must} be canceled by other errors if the Ratios Conjecture is to yield the correct prediction.
\end{rek}

%%%%%%%%%%%%%%%%%%%%%%%%%%%%%%%%%%%%%%%%%%%%%%%%%%%%%%%%%%%%%%%%%%%%%%%%%%%%%
%%%%%%%%%%%%%%%%%%%%%%%%%%%%%%%%%%%%%%%%%%%%%%%%%%%%%%%%%%%%%%%%%%%%%%%%%%%%%

\section{Weighted 1-level density from the Ratios Conjecture}\label{sec:1ldfromratiosconjexpansiontheorems}

\subsection{Main Expansion}\label{sec:oneldfromratiosconjmainexpansion}

We now compute the $1$-level density for the family $H_k^\ast(N)$,
with either $N=1$ and $k\to\infty$ or $k$ a fixed even integer and
$N$ tending to infinity through the primes. We follow closely the
arguments in \cite{CS,Mil4}.

\begin{lem}\label{lem:ratiosconj1level1}
Assume GRH for $\zeta(s)$ and all
$L(s,f)$ with $f\in H_k^\ast(N)$. Denote the weighted $1$-level
density for the family $H_k^\ast(N)$ by \be
D_{1,H_k^\ast(N);R}(\phi) \ = \ \sum_{f\in H_k^\ast(N)} \omega_f^\ast(N)
\sum_{\gamma_f \atop L(1/2+i\gamma_f,f) = 0} \phi\left(\gamma_f
\frac{\log R}{2\pi}\right). \ee Assuming the Ratios Conjecture, we have
\bea\label{eq:d1hkastnrphitemp1} & & D_{1,H_k^\ast(N);R}(\phi)  \ =
\ 2\sum_p \hphi\left(\frac{2 \log p}{\log R}\right) \frac{\log
p}{p\log R} \nonumber\\ & & \ \ \ \ \ +\ \frac{2i^k \mu(N)}{N\log R}
\int_{-\infty}^\infty X_L\left(\foh + \frac{2\pi i t}{\log R}\right)
\prod_{p \neq N} \left(1 + \frac{1}{(p-1)p^{4\pi it/\log R}}\right)
e^{-2\pi i t \frac{\log N}{\log R}}
\phi(t) dt \nonumber\\
& &  \ \ \ \ \ +\ \frac1{\log R} \int_{-\infty}^\infty \left(2 \log
\frac{\sqrt{N}}{\pi} + \psi\left(\frac14 + \frac{k\pm 1}4
+\frac{2\pi i t}{\log R}\right) \right) \phi(t) dt \nonumber\\ & & \ \ \ \ \ + \ O\left((kN)^{-1/2+\gep}\right). \eea
\end{lem}

\begin{proof} We first compute the unscaled, weighted $1$-level density
$S_{1;H_k^\ast(N)}(g)$ for the family $H_k^\ast(N)$ with $g$ an even
Schwartz function, \be S_{1;H_k^\ast(N)}(g) \ = \ \sum_{f\in
H_k^\ast(N)} \omega_f^\ast(N) \sum_{\gamma_f \atop L(1/2+i\gamma_f,f) = 0}
g(\gamma_f). \ee Let $c \in \left(\foh+\frac1{\log k^2 N},
\frac34\right)$; thus \bea S_{1;H_k^\ast(N)}(g) & \ = \ & \sum_{f\in
H_k^\ast(N)} \frac1{2\pi i} \left(\int_{(c)} - \int_{(1-c)}\right)
\omega_f^\ast(N) \frac{L'(s,f)}{L(s,f)}
g\left(-i\left(s-\foh\right)\right)ds \nonumber\\ &=&
S_{1,c;H_k^\ast(N)}(g) + S_{1,1-c;H_k^\ast(N)}(g). \eea We argue as
on page 15 of \cite{CS}. We first analyze the integral on the line
$\Re(s) = c$. By GRH and the rapid decay of $g$, for large $t$ the
integrand is small. We use the Ratios Conjecture (Lemma
\ref{lem:Rprimealphagammar} with $r = c - \foh + it$) to replace the
$\sum_f \omega_f^\ast(N) L'(s,f)/L(s,f)$ term when $t$ is small. We may then extend the integral to all of $t$ because of the
rapid decay of $g$. As the integrand is regular at $r=0$ we can move
the path of integration to $c=1/2$. The contribution from the error term in the Ratios Conjecture is negligible, due to $g$ being a Schwartz function. Thus the integral on the
$c$-line is \bea S_{1,c;H_k^\ast(N)}(g) & \ = \  & \frac{1}{2\pi i}
\int_{-\infty}^\infty g\left(t - i\left(c-\frac12\right)\right)
\sum_{f\in H_k^\ast(N)}\omega_f^\ast(N) \frac{L'(\foh +
(c-\foh+it),f)}{L(\foh + (c-\foh+it),f)} idt \nonumber\\ & = & \
\frac1{2\pi} \int_{-\infty}^\infty g(t) \Bigg[ \sum_p \frac{\log
p}{p^{1+2it}} \nonumber\\ & & \ \ \  + \frac{i^k\mu(N)}{N^{1+it}}
X_L\left(\foh + it\right) \prod_{p} \left(1 +
\frac{1}{(p-1)p^{2it}}\right)\Bigg]dt. \nonumber\\ & & \ + \ O\left((kN)^{-1/2+\gep}\right). \eea As \be \int_{-\infty}^\infty g(t)
p^{-2it}dt \ = \ \int_{-\infty}^\infty g(t) e^{-2\pi i (\frac{2\log
p}{2\pi}) t} dt \ = \ \widehat{g}\left(\frac{2\log p}{2\pi}\right)
\ee we have  \bea & & S_{1,c;H_k^\ast(N)}(g)  \ = \
\frac1{2\pi}\sum_p \widehat{g}\left(\frac{2\log p}{2\pi}\right)
\frac{\log p}{p} \nonumber\\ & &\ \ +\ \frac{i^k \mu(N)}{2\pi N}
\int_{-\infty}^\infty X_L\left(\foh + it\right) \prod_{p}
\left(1 + \frac{1}{(p-1)p^{2it}}\right) N^{-it} g(t) dt\nonumber\\ & & \ \ + \ O\left((kN)^{-1/2+\gep}\right). \eea

We now study $S_{1,1-c;H_k^\ast(N)}(g)$: \be
S_{1,1-c;H_k^\ast(N)}(g) \ = \ \sum_{f\in H_k^\ast(N)}
\frac{-\omega_f^\ast(N)}{2\pi i}  \int_{\infty}^{-\infty}
\frac{L'(1-(c+it),f)}{L(1-(c+it),f)}
g\left(-i\left(\foh-c\right)-t\right) (-idt). \ee We use the
functional equation \be L(s,f) \ = \ \epsilon_f X_L(s) L(1-s,f) \ee
to find that \be \frac{L'(1-(c+it),f)}{L(1-(c+it),f)} \ = \ -
\frac{L'(c+it,f)}{L(c+it,f)} + \frac{X_L'(c+it)}{X_L(c+it)}. \ee
This yields \bea S_{1,1-c;H_k^\ast(N)}(g) & \ = \  & \frac1{2\pi}
\int_{-\infty}^\infty \sum_{f\in H_k^\ast(N)} \omega_f^\ast(N)
\frac{L'(c+it,f)}{L(c+it,f)} g\left(-i\left(\foh-c\right)-t\right)
dt \nonumber\\ & & - \ \frac1{2\pi} \int_{-\infty}^\infty
\frac{X_L'(c+it)}{X_L(c+it)}g\left(-i\left(\foh-c\right)-t\right)dt.
\eea The first term yields the same contribution as
$S_{1,c;H_k^\ast(N)}(g)$; this follows by sending $c$ to $1/2$ and
noting $g$ is an even function. Thus \bea S_{1;H_k^\ast(N)}(g) & \ =
\ & \frac2{2\pi}\sum_p \widehat{g}\left(\frac{2\log p}{2\pi}\right)
\frac{\log p}{p} \nonumber\\ & & \ + \frac{2i^k \mu(N)}{2\pi N}
\int_{-\infty}^\infty X_L\left(\foh + it\right) \prod_{p}
\left(1 + \frac{1}{(p-1)p^{2it}}\right)
N^{-it} g(t) dt \nonumber\\ & & - \ \frac1{2\pi}
\int_{-\infty}^\infty \frac{X_L'(1/2+it)}{X_L(1/2+it)}g(t)dt + O\left((kN)^{-1/2+\gep}\right). \eea

In investigating zeros near the central point, it is convenient to
renormalize them by the logarithm of the analytic conductor. Let
$g(t) = \phi\left(\frac{t \log R}{2\pi}\right)$. A straightforward
computation shows that $\widehat{g}(\xi) = \frac{2\pi}{\log R}
\hphi(2\pi \xi / \log R)$. The (scaled) weighted $1$-level density for the
family $H_k^\ast(N)$ is \be D_{1,H_k^\ast(N);R}(\phi) \ = \
\sum_{f\in H_k^\ast(N)} \omega_f^\ast(N) \sum_{\gamma_f \atop
L(1/2+i\gamma_f,f) = 0} \phi\left(\gamma_f \frac{\log
R}{2\pi}\right) \ = \ S_{1;H_k^\ast(N)}(g) \ee (where $g(t) =
\phi\left(\frac{t \log R}{2\pi}\right)$ as before). Thus \bea & &
D_{1,H_k^\ast(N);R}(\phi)  \ = \  2\sum_p \hphi\left(\frac{2 \log
p}{\log R}\right) \frac{\log p}{p\log R} \nonumber\\ & & \ \ \ \ \ +
\frac{2i^k\mu(N)}{2\pi N}  \int_{-\infty}^\infty X_L\left(\foh +
it\right) \prod_{p \neq N} \left(1 + \frac{1}{(p-1)p^{2it}}\right)
N^{-it} \phi\left(\frac{t \log R}{2\pi}\right) dt \nonumber\\
& &  \ \ \ \ \ - \ \frac1{2\pi} \int_{-\infty}^\infty
\frac{X_L'(1/2+it)}{X_L(1/2+it)} \phi\left(\frac{t \log
R}{2\pi}\right)dt + O\left((kN)^{-1/2+\gep}\right). \eea Changing variables yields \bea\label{eq:tempeqd1hkastNR} & &
D_{1,H_k^\ast(N);R}(\phi)  \ = \  2\sum_p \hphi\left(\frac{2 \log
p}{\log R}\right) \frac{\log p}{p\log R} \nonumber\\ & & \ \ \ \ \ +
\frac{2i^k \mu(N)}{N\log R} \int_{-\infty}^\infty X_L\left(\foh +
\frac{2\pi i t}{\log R}\right) \prod_{p \neq N} \left(1 +
\frac{1}{(p-1)p^{4\pi it/\log R}}\right)
e^{-2\pi i t \frac{\log N}{\log R}}
\phi(t) dt \nonumber\\
& &  \ \ \ \ \ - \ \frac1{\log R} \int_{-\infty}^\infty
\frac{X_L'\left(\foh+\frac{2\pi it}{\log
R}\right)}{X_L\left(\foh+\frac{2\pi it}{\log R}\right)}\ \phi(t)dt + O\left((kN)^{-1/2+\gep}\right).
\eea

Set $\psi(z) = \Gamma'(z)/\Gamma(z)$. As the derivative of $\log
X_L(s)$ is $X_L'(s)/X_L(s)$, we find \bea
-\frac{X_L'\left(\foh+\frac{2\pi it}{\log
R}\right)}{X_L\left(\foh+\frac{2\pi it}{\log R}\right)} & \ = \ &
2\log \frac{\sqrt{N}}{\pi} + \foh \psi\left(\frac14 + \frac{k\pm 1}4
\pm \frac{2\pi i t}{\log R}\right) \eea (note there are four
$\psi$-terms). As $\phi$ is an even function, the $+t$ and $-t$
terms yield the same integral, completing the proof. \end{proof}

The first sum and the last integral in Lemma
\ref{lem:ratiosconj1level1} will match up perfectly with terms from
the number theory calculation. In \S\ref{sec:proofratiosconj1level1} we finish the proof of Theorem \ref{thm:ratiosconj1level1} by analyzing the middle term.

\subsection{Proof of Theorem \ref{thm:ratiosconj1level1}}\label{sec:proofratiosconj1level1}

\begin{proof}[Proof of Theorem \ref{thm:ratiosconj1level1}]
Most of the analysis for the first part of the theorem has been done in \S\ref{sec:oneldfromratiosconjmainexpansion}; in particular, the expansion in \eqref{eq:tempeqd1hkastNR}. The proof is completed by Lemmas \ref{lem:newprodforprimesinmiddleterm} and \ref{lem:Mphi} below, which derive a simpler expression for the middle piece and then show it yields a negligible contribution. \end{proof}

\begin{lem}\label{lem:newprodforprimesinmiddleterm}
Let $\Re(u) = 0$. Then \be \prod_p \left(1 + \frac1{(p-1)p^u}\right)
\ = \ \frac{\zeta(2)}{\zeta(2+2u)} \cdot \zeta(1+u) \cdot \prod_p
\left(1 - \frac{p^u-1}{p(p^{1+u}+1)}\right); \ee note the product
over primes converges rapidly for $\Re(u) = 0$, as each term in the
product is like $1 + O(1/p^2)$. \end{lem}

\begin{proof}
We have \bea\label{eq:rewritingprodprimesmiddleterm} \prod_p \left(1
+ \frac1{(p-1)p^u}\right) & \ = \ & \prod_p \left(1 +
\frac1{p^{1+u}}
\right) \cdot \left(1 + \frac{1}{(p-1)(p^{1+u}+1)}\right)\nonumber\\
&=& \prod_p\frac{\left(1 + \frac1{p^{1+u}} \right) \cdot \left(1 -
\frac1{p^{1+u}}\right)}{1 - \frac1{p^{1+u}}} \cdot \left(1 +
\frac{1}{(p-1)(p^{1+u}+1)}\right) \nonumber\\ & = &
\frac{\zeta(1+u)}{\zeta(2+2u)} \cdot \prod_p\left(1 +
\frac{1}{(p-1)(p^{1+u}+1)}\right). \eea We can rewrite this a little
further, using \bea  \prod_p\left(1 +
\frac{1}{(p-1)(p^{1+u}+1)}\right) & \ = \ & \prod_p
\frac{p^2}{p^2-1} \left(1 - \frac{p^u-1}{p(p^{1+u}+1)}\right)
\nonumber\\ & = & \prod_p \frac1{1 - \frac1{p^2}} \cdot  \left(1 -
\frac{p^u-1}{p(p^{1+u}+1)}\right) \nonumber\\ & = & \zeta(2) \prod_p
\left(1 - \frac{p^u-1}{p(p^{1+u}+1)}\right). \eea Substituting this
into \eqref{eq:rewritingprodprimesmiddleterm} completes the proof.
\end{proof}

\begin{rek} When arguing along the lines of the Ratios Conjecture, it often greatly simplifies the calculations to rewrite the prime products in a more rapidly convergent manner by factoring out zeta or $L$-functions. In Lemma \ref{lem:xlgammafactors} we use the above expansion to show that the $X_L$ term in the 1-level density is negligible. When $N=1$ this is the hardest part of the proof, and follows by shifting contours. \end{rek}

\begin{lem}\label{lem:Mphi} Let \bea M(\phi) & \ = \ &  \frac{2i^k \mu(N)}{N\log
R} \int_{-\infty}^\infty X_L\left(\foh + \frac{2\pi i t}{\log
R}\right) \frac{\zeta(2)}{\zeta(2+\frac{8\pi i t}{\log
R})}\zeta\left(1+\frac{4\pi i t}{\log R}\right) \nonumber\\ & & \ \
\ \ \ \cdot \ \prod_p \left(1 - \frac{p^{4\pi i t/\log
R}-1}{p(p^{1+4\pi i t/\log R}+1)}\right) e^{-2\pi i t \frac{\log N}{\log R}} \phi(t) dt. \eea If $N > 1$ we have $M(\phi) = O(1/N)$. Assume $\supp(\hphi) \subset (-\sigma, \sigma)$. If $N = 1$ then $M(\phi) = O\left(2009^{k} \cdot k^{-\frac{1-4\sigma}{3}k}\right)$, which tends to zero more rapidly than $k^{-\delta}$ for any $\delta > 0$ for $\sigma < 1/4$. \end{lem}

\begin{proof} We use the lemmas from \S\ref{sec:usefulestimates} to bound the relevant quantities. As $\phi$ is an even function, there is no contribution from the pole of the Riemann zeta function.

Assume first $N > 1$. If $u \ge 0$ then \be\label{eq:valuezetafnat2plus2u} \left|\zeta\left(2+2u+\frac{8\pi i t}{\log R}\right)\right| \ \ge \ 1 - \sum_{n=2}^\infty \frac1{n^2} \ = \ 1 - \left(\frac{\pi^2}6-1\right) > 0.\ee As remarked above, there is no contribution from the pole of the Riemann zeta function (since $\phi$ is even). We may thus subtract off the pole without changing the value of the integral, and note that  \be \left|\zeta\left(1+\frac{4\pi i t}{\log R}\right) - \frac{\log R}{4\pi i t}\right|  \ \ll \ (t^2+1) \log R \ee (we could of course do far better, but a very weak bound suffices for large $t$ due to the rapid decay of $\phi$). Thus the product of the zeta terms is $O((t^2+1)\log N)$. The product over primes is bounded by \be \prod_p \left(1 - \frac{2}{p(p-1)}\right), \ee which is $O(1)$. Finally, the $X_L$-term is $O(1)$ by Lemma \ref{lem:xlgammafactors}. Thus \be M(\phi) \ \ll \  \frac1{N\log R} \int_{-\infty}^\infty (t^2 + 1) \log R \cdot \phi(t) dt \ \ll \ \frac1N  \ee (as $\phi$ is a Schwartz function).

Assume now that $N=1$. We follow the method used in \cite{Mil4}, and replace $t$ with $t - iw \frac{\log R}{4\pi}$ (where
initially $w=0$), shift contours and exploit the decay in $w$.
By analyzing $X_L$ and the zeta factors, we see we may shift the contour
to $w=2k-1-\epsilon$ without passing through any
zeros or poles. We shift to $w=\frac{2k-1}3$ as this will simplify some of the computations. We have \bea M(\phi) & \ = \ & - \frac{2i^k}{\log R}
\int_{-\infty}^\infty X_L\left(\frac{1+w}2 + \frac{2\pi i t}{\log
R}\right) \frac{\zeta(2)}{\zeta(2+2w+\frac{8\pi i t}{\log
R})}\zeta\left(1+w+\frac{4\pi i t}{\log R}\right) \nonumber\\ & & \
\ \ \ \ \cdot \ A\left(w+\frac{2\pi i t}{\log R}\right) \cdot
\phi\left(t-iw\frac{\log R}{2\pi}\right) dt\nonumber\\ \eea where \be A(x+iy) \ = \ \prod_p
\left(1 - \frac{p^{x+iy}-1}{p(p^{1+x+iy}+1)}\right). \ee As $\phi$ is even, there is no contribution from the pole of the zeta function. In the arguments below, we could be more explicit and subtract off this pole. The shifted term will have a factor of size $O\left((w^2+(t/\log R)^2)^{-1}\right) = O(1)$, which will not change any of the arguments.

From Lemma \ref{lem:sizeprodpfnofu} we have $A\left(w+\frac{2\pi i t}{\log R}\right) = O(1)$. For any $w > 0$, by \eqref{eq:valuezetafnat2plus2u} the ratio of the zeta factors $\zeta(2)/\zeta(2+2w+\frac{8\pi i t}{\log
R})$ is $O(1)$. From Lemma \ref{lem:xlgammafactors} we know that for $w = \frac{2k-1}3$ the $X_L$-term is $O\left(2009^{k} \cdot k^{-k/3}\right)$, and from Lemma \ref{lem:decayphi} we have \be\label{eq:boundphicontourshiftw} \phi\left(t-iw \frac{\log R}{2\pi}\right) \ \ll \ \exp\left( \sigma w \log R \right) \cdot \left(t^2 + \frac{\log^2 R}{16\pi^2}\right)^n \ \ll \ \frac{R^{\sigma w}}{(t^2 + 1)^n}. \ee Thus \bea M(\phi) & \ \ll \ & \left(2009^{k} \cdot k^{-k/3}\right) \cdot  \frac{R^{\sigma w}}{\log R} \int_{-\infty}^\infty \frac{dt}{(t^2+1)^n} \ \ll \ \frac{2009^k R^{\sigma w}}{k^{k/3}\log R}. \eea For cuspidal newforms of level 1 and weight $k$, one takes (see (1.14) and (4.29) of \cite{ILS}) $R \sim k^2$. As $w = \frac{2k-1}3$, the above decays more rapidly than \be 2009^k \cdot k^{\frac{4k\sigma}3-\frac{k}3} \ = \ 2009^k \cdot k^{-\frac{1-4\sigma}{3}k}; \ee thus as long as $\sigma < 1/4$, this term decays faster than $k^{-\delta}$ for any $\delta > 0$.
\end{proof}

%\begin{rek} In the proof of Lemma \ref{lem:Mphi} we could have %handled the $N>1$ case by shifting contours (as we did for $N=1$). %While this would improve the error term, we prefer to give the %elementary proof as the error is already small ($O(1/N)$). %\end{rek}

\begin{rek}\label{rek:whyn=1harderntoinfinity} Note the results in Lemma \ref{lem:Mphi} are significantly worse for $N=1$ than for $N \to\infty$. This is due to the rapid growth of $\hphi(x+iy)$ in $y$, and leads to a significantly reduced support. This is very similar to the difficulties encountered in studying families of quadratic characters \cite{Mil4}, where we again had to perform a contour shift, which restricted our results to $\sigma < 1$ (with square-root agreement for $\sigma < 1/3$). Our result is weaker than the corresponding result in \cite{Mil4} (we have $\sigma < 1/4$ instead of $\sigma < 1$) because here the conductor is $k^2$ (whereas in \cite{Mil4} the conductor is $d$) \emph{and} $k$ appears in the Gamma factors.
\end{rek}

\begin{rek} Another approach to analyzing $M(\phi)$ when $N=1$ is to shift the contour \emph{very} far to the right, picking up contributions from the poles of the Gamma function in the numerator of $X_L$. Unfortunately the resulting expressions can only be shown to be small for $\sigma < 1/4$. The poles arise when $w = k - \foh + 2\ell$ for $\ell \in \{0,1,2,\dots\}$, and yield contributions of \be \frac{-2i^k}{\log R} \frac{\zeta(2) \zeta\left(1+k-\foh+2\ell\right) A\left(k-\foh+2\ell\right)}{\zeta(2+2k-1+4\ell) \Gamma\left(k-\foh+\ell\right)}\cdot \left(\frac{2\pi}{\sqrt{N}}\right)^{k-\foh+2\ell} \cdot \phi\left(\frac{-i(k-\foh+2\ell)\log R}{2\pi}\right). \ee The $\phi$ term is at most $R^{\sigma(k-\foh+2\ell)}$, while the main term of $\Gamma(k-\foh+\ell)$ is of size $k^\ell k^{k-\foh+\ell}$. The problem is the resulting sum over $\ell$ is only small if $\sigma < 1/4$, though based on our number theory computations we expect it to be small for $\sigma < 1$ or even $\sigma < 2$. The difficulty is that we are ignoring all oscillation when we shift contours.
\end{rek}

%%%%%%%%%%%%%%%%%%%%%%%%%%%%%%%%%%%%%%%%%%%%%%%%%%%%%%%%%%%%%%%%%%%%%%
%%%%%%%%%%%%%%%%%%%%%%%%%%%%%%%%%%%%%%%%%%%%%%%%%%%%%%%%%%%%%%%%%%%%%%

\section{Weighted 1-level density from Number
Theory}\label{sec:onelevelfromnt}

We now determine the main and lower order terms in the 1-level density for the family $H_k^\ast(N)$ for as large of support as possible for the Fourier transform of the test function. In \cite{ILS} the main term is determined for $\supp(\hphi) \subset (-2, 2)$; however, as they are only concerned with the main term they are a little crude in bounding the error terms. We perform a more careful analysis below.

In Section 4 of \cite{ILS} the explicit formula is used to compute
the $1$-level density for the family $H_k^\ast(N)$. In their paper
$Q = \sqrt{N}/\pi$. Noting that $\hphi(0) = \int_{-\infty}^\infty
\phi(t)dt$, we may rewrite the weighted sum over $f \in H_k^\ast(N)$
of their equation (4.11) as \bea D_{1,H_k^\ast(N);R}(\phi) & \ = \ &
\frac1{\log R} \int_{-\infty}^\infty \left(2 \log
\frac{\sqrt{N}}{\pi} + \psi\left(\frac14 + \frac{k\pm 1}4
+\frac{2\pi i t}{\log R}\right) \right) \phi(t) dt \nonumber\\ & & \
\ \ - \ 2\sum_{f \in H_k^\ast(N)} \omega_f^\ast(N) \sum_p \sum_{\nu =
1}^\infty \frac{\alpha_f^\nu(p)+\beta_f^\nu(p)}{p^{\nu/2}}
\hphi\left(\frac{\nu\log p}{\log R}\right) \frac{\log p}{\log
R},\nonumber\\ \eea where \be L(s,f) \ = \ \sum_{n=1}^\infty
\frac{\lambda_f(n)}{n^s} \ = \ \prod_p \left(1 -
\frac{\alpha_f(p)}{p^s}\right)^{-1} \left(1 -
\frac{\beta_f(p)}{p^s}\right)^{-1}. \ee Note the first term agrees
exactly with the last term from the Ratios Conjecture (Theorem \ref{thm:ratiosconj1level1}).

The following identities for the Fourier coefficients (for $p\notdiv
N$) are standard: \bea\label{eq:standardformslambdafp} \lambda_f(p) & \ = \ & \alpha_f(p) +
\alpha_f(p)^{-1}, \ \ \ |\alpha_f(p)|\ =\ 1, \ \ \ \alpha_f(p)^{-1} \ = \ \beta_f(p) \nonumber\\
\lambda_f(p^\nu) & = & \alpha_f(p)^\nu + \alpha_f(p)^{\nu-2} +
\cdots + \alpha_f(p)^{2-\nu} + \alpha_f(p)^{-\nu} \nonumber\\
\alpha_f(p)^\nu + \alpha_f(p)^{-\nu} & = & \lambda_f(p^\nu) -
\lambda_f(p^{\nu-2}). \eea

Trivially bounding the contribution from $p=N$, we may thus rewrite $D_{1,H_k^\ast(N);R}(\phi)$ as \bea\label{eq:NTexpansion1ldHknast}
D_{1,H_k^\ast(N);R}(\phi) & \ = \ & \frac1{\log R}
\int_{-\infty}^\infty \left(2 \log \frac{\sqrt{N}}{\pi} +
\psi\left(\frac14 + \frac{k\pm 1}4 +\frac{2\pi i t}{\log R}\right)
\right) \phi(t) dt \nonumber\\ & & \ \ \ - \ 2\sum_{f \in
H_k^\ast(N)} \omega_f^\ast(N) \sum_{p\neq N} \frac{\lambda_f(p)}{\sqrt{p}}\
\hphi\left(\frac{\log p}{\log R}\right) \frac{\log p}{\log R}
\nonumber\\ & & \ \ \ - 2\sum_{f \in H_k^\ast(N)} \omega_f^\ast(N)
\sum_{p\neq N} \frac{\gl_f(p^2)-1}{p}\ \hphi\left(2\frac{\log p}{\log
R}\right) \frac{\log p}{\log R} \nonumber\\ & & \ \ \ - 2\sum_{f \in
H_k^\ast(N)} \omega_f^\ast(N) \sum_{p\neq N} \sum_{\nu = 3}^\infty
\frac{\lambda_f(p^\nu) - \lambda_f(p^{\nu-2})}{p^{\nu/2}}\
\hphi\left(\nu\frac{\log p}{\log R}\right) \frac{\log p}{\log
R}\nonumber\\   & & \ \ \ + O\left(\frac1{\sqrt{N}}\right) \nonumber\\ & = & \frac1{\log R}
\int_{-\infty}^\infty \left(2 \log \frac{\sqrt{N}}{\pi} +
\psi\left(\frac14 + \frac{k\pm 1}4 +\frac{2\pi i t}{\log R}\right)
\right) \phi(t) dt
\nonumber\\ & & \ \ \ + 2
\sum_p \frac{1}{p}\ \hphi\left(2\frac{\log p}{\log
R}\right) \frac{\log p}{\log R} \nonumber\\ & & \ \ \ - S_1(\phi) - S_2(\phi) - S_3(\phi) + O\left(\frac1{\sqrt{N}}\right),  \eea where \bea S_1(\phi) & \ = \ & 2\sum_{f \in
H_k^\ast(N)} \omega_f^\ast(N) \sum_{p\neq N} \frac{\lambda_f(p)}{\sqrt{p}}\
\hphi\left(\frac{\log p}{\log R}\right) \frac{\log p}{\log R} \nonumber\\ S_2(\phi)& \ = \ & 2 \sum_{f \in
H_k^\ast(N)} \omega_f^\ast(N) \sum_{p\neq N} \frac{\gl_f(p^2)}{p} \hphi\left(2\frac{\log p}{\log R}\right) \frac{\log p}{\log R}
\nonumber\\ S_3(\phi) & \ = \ & 2\sum_{f \in
H_k^\ast(N)} \omega_f^\ast(N) \sum_{p\neq N} \sum_{\nu = 3}^\infty
\frac{\lambda_f(p^\nu) - \lambda_f(p^{\nu-2})}{p^{\nu/2}}\
\hphi\left(\nu\frac{\log p}{\log R}\right) \frac{\log p}{\log
R}.\ \ \ \eea

The first and the second terms above perfectly match with terms from the Ratios Conjecture. We must therefore show the other three terms are negligible. We prove Theorems \ref{thm:1ldnumbthN} and \ref{thm:1ldnumbthk} in stages below; we first perform the analysis for limited support, and then extend the support by assuming various conjectures.

\subsection{Density Theorem Limited}

As the arguments are similar when $N\to\infty$ through the primes and when $N=1$ and $k\to\infty$, we give complete details for $N\to\infty$ and sketch the arguments when $N=1$.

\begin{rek}
It is important to note that we have included the harmonic (or Petersson) weights in our family to facilitate applications of the Petersson formula. When using results from \cite{ILS}, one must be careful as they have three related quantities involving averages of the Fourier coefficients over families. The first (converting to our notation) is their equation (2.7), \be \Delta_k(m,n) \ = \ \sum_{f \in \mathcal{B}_k(N)} \omega_f(N) \gl_f(m) \gl_f(n); \ee the weights sum to 1, and thus in this expression we have effectively divided by the cardinality of the family. Note that we are summing over all cusp forms of weight $k$ and level $N$, and not just the newforms. The second is their equation (2.54), where we sum over just the newforms: \be \Delta_{k,N}^\sigma(m,n) \ = \ \zeta(2) \sum_{f \in H_k^\sigma(N)} \frac{\gl_f(m) \gl_f(n)}{L(1,{\rm sym}^2 f)}, \ \ \ \sigma \in \{\ast,+,-\}. \ee Finally, we have the unweighted, pure sums (their equation (2.59)): \be \Delta_{k,N}^\sigma(n) \ = \ \sum_{f\in H_k^\sigma(N)} \gl_f(n), \ \ \ \sigma \in \{\ast,+,-\}. \ee Much effort was spent in \cite{ILS} to remove the weights; thus when reading their paper we must look carefully to see which variant they are using. \end{rek}

\begin{lem} Let $\supp(\hphi) \subset (-\sigma, \sigma)$. Then $S_1(\phi) \ll N^{\sigma-\frac32+\gep}+ N^{\frac{\sigma}2-1+\gep}$ as $N\to\infty$ through the primes, and if $\sigma < 1$ then $S_1(\phi) \ll k^{2\sigma} 2^{-k}$ for $N=1$ and $k\to\infty$.
\end{lem}

\begin{proof} Assume $N > 1$ tends to infinity through the primes. We use the Petersson formula (Lemma \ref{lem:Peterssonjustnewforms}) to bound the weighted sum of $\lambda_f(p)$, and find \bea S_1(\phi) & \ \ll \ & 2 \sum_{p \neq N \atop p \le R^\sigma} \frac{\log R}{N\sqrt{p}} \left( \frac{\sqrt{p}}{\sqrt{N+\sqrt{p}}} + (pN)^\gep \right). \eea As  $1/\sqrt{N + \sqrt{p}} \ll  1/\sqrt{N}$, we find \bea S_1(\phi) & \ \ll \ &  \sum_{p \le R^\sigma} \frac{\log N}{N\sqrt{N}} + N^{\gep'+\frac{\sigma}2-1} \ \ll \ N^{\sigma-\frac32+\gep} + N^{\frac{\sigma}2-1+\gep}. \eea If now $N=1$ and $k\to\infty$, we use Lemma \ref{lem:ils23} (which forces us to take $\sigma < 1$ as $R = k^2$) and find \be S_1(\phi) \ \ll \ \frac{1}{2^k} \sum_{p \le k^{2\sigma}} \frac{\log p}{\log R} \ \ll \ k^{2\sigma} 2^{-k}, \ee which is  $O(k^{-1/2})$ for $k$ large.
\end{proof}

\begin{rek} If $\sigma < 1$, then $S_1(\phi) \ll N^{-1/2}$ or $k^{-1/2}$, and we obtain square-root agreement of this term with the Ratios prediction (if $N=1$ we must restrict to $\sigma < 1/4$ because of our estimate for $M(\phi)$). For $\sigma \ge 1$ we don't have such phenomenal agreement (we can take $\sigma < 3/2$ for $N\to\infty$, but if $k\to\infty$ the above arguments fail for $\sigma \ge 1$), but we do at least agree up to a power of $N$. We have not exploited any cancelation in the Bessel-Kloosterman terms (we shall do this in \S\ref{sec:densitytheoremextended}), contenting ourselves here to argue simply and crudely. The quality of our results is exactly the same as that in Theorem 5.1 of \cite{ILS} (where they have not yet exploited properties of the Bessel-Kloosterman terms, which is required to increase the support).
\end{rek}

\begin{lem}\label{lem:S2S3Nk} Let $\supp(\hphi) \subset (-\sigma, \sigma)$. \ben \item  We have $S_2(\phi) \ll N^{\frac{\sigma}{4}-1+\gep''}$ as $N\to\infty$ through the primes, and $S_2(\phi) \ll k^{-(5-3\sigma)/6+\gep}$ if $N=1$ and $k\to\infty$.

\item We have $S_3(\phi) \ll N^{\frac{\sigma}{12}-1+\gep''}$ as $N\to\infty$ through the primes, and $S_3(\phi) \ll k^{-(5-\sigma)/6+\gep}$ if $N=1$ and $k\to\infty$.

\een

\end{lem}

\begin{proof} As the proofs are similar, we only prove the second statement. We first consider $N\to\infty$. We apply the Petersson Formula (Lemma \ref{lem:Peterssonjustnewforms}) to the sums of $\lambda_f(p^{\nu})$ and $\lambda_f(p^{\nu-2})$. As the error from the $\lambda_f(p^{\nu-2})$ terms is dominated by the error from the $\lambda_f(p^\nu)$ terms, we only consider the former. As we evaluate $\hphi$ at $\nu\log p / \log R$ with $n \ge 3$, we may restrict the $p$-sums to $p \le R^{\sigma/3}$ (where $R = k^2N$).  We find \bea S_3(\phi) & \ \ll \ & \sum_{(p, N) = 1 \atop p \le R^{\sigma/3}} \sum_{\nu=3}^{\log_p R} \frac1{p^{\nu/2}} \left(\frac{\log N}{N} \frac{p^{\nu/2}}{\sqrt{N + p^{\nu/2}}} + \frac{(p^\nu N)^\gep}{N}\right) \nonumber\\ & \ll & \frac{\log^2 N}{N} \sum_{p \le R^{\sigma/3}} p^{-\frac34} + N^{\gep'-1} \nonumber\\ & \ll & N^{\frac{\sigma}{12}-1+\gep} + N^{\gep'-1} \ \ll \ N^{\frac{\sigma}{12}-1+\gep''}. \eea We now examine the case when $N=1$ and $k\to\infty$. We use Lemma \ref{lem:ilscor22}. As $R = k^2$ and $\nu \ge 3$, the prime sum is restricted to $p \le k^{2\sigma/3}$. We find \be S_3(\phi) \ \ll \ \log^2 k \sum_{p \le k^{2\sigma/3}} \frac{1}{k^{5/6}} \frac{1}{\sqrt{p^{3/2}+k}} \ \ll \ k^{-5/6+\gep} \sum_{p \le k^{2\sigma/3}} p^{-3/4} \ \ll \ k^{-(5-\sigma)/6+\gep}.\ee
\end{proof}

\begin{rek} Even for $\sigma < 6$ (which is \emph{well} beyond current technology for analyzing $S_1(\phi)$!), $S_3(\phi)$ is $O(N^{-1/2})$; it is $O(k^{-1/2})$ for $\sigma < 2$, which \emph{is} in the range of current technology. If $N>1$ then $S_2(\phi) = O(N^{-1/2+\gep})$ for $\sigma <2$; however, if $N=1$ then we only have square-root cancelation up to $\sigma = 2/3$ (in fact, if $\sigma \ge 5/3$ then our argument is too crude to bound this term). Thus the difficulty in showing agreement between number theory and the Ratios Conjecture's predictions is entirely due to $S_1(\phi)$ on the number theory side and $M(\phi)$ on the Ratios side. \end{rek}

\subsection{Density Theorem Extended}\label{sec:densitytheoremextended}

To improve our 1-level density results for $H_k^\ast(N)$, we need to improve our analysis of \be S_1(\phi) \ = \ 2\sum_{f \in
H_k^\ast(N)} \omega_f^\ast(N) \sum_{p\neq N} \frac{\lambda_f(p)}{\sqrt{p}}\
\hphi\left(\frac{\log p}{\log R}\right) \frac{\log p}{\log R}. \ee We are able to show agreement with the Ratios Conjecture up to a power savings in $N$ if $\supp(\hphi) \subset (-\sigma, \sigma)$ with $\sigma < 2$ (with additional analysis of $S_2(\phi)$ we should be able to extend our results up to $\sigma < 2$ when $N=1$). To do this we modify the arguments in \cite{ILS}. There are two major differences. First, they were concerned only with the main term and $N$ square-free, and thus some of their error terms can be significantly improved for $N$ prime. Second, they studied the unweighted sum (i.e., they did not include the Petersson weights). Including the Petersson weights simplifies the computations, though they can be done with the unweighted sum as well (see \S\ref{sec:extendingsuppunweightedS1}).

\begin{lem}\label{lem:s1phiksum} Assume GRH for $\zeta(s)$, all Dirichlet $L$-functions and all $L(s,f)$ with $f \in S_k(N)$. If $N\to\infty$ through the primes then $S_1(\phi) \ll N^{\frac{\sigma}2-1+\gep}$.
\end{lem}

\begin{proof} The most difficult part in the proofs in \cite{ILS} were from handling the non-diagonal terms in the unweighted Petersson formula. We bypass some of these difficulties by using weighted sums.
We have \be S_1(\phi) \ = \  \sum_{p\neq N} \left(\sum_{f \in
H_k^\ast(N)} \omega_f^\ast(N)\lambda_f(p) \right)
\hphi\left(\frac{\log p}{\log R}\right) \frac{2\log p}{\sqrt{p}\log R}. \ee Let \bea Q_k^\ast(m;c) \ = \ 2\pi i^k \sum_{p\notdiv N} \frac{S(m^2,p;c)}{c} J_{k-1}\left(\frac{4\pi m\sqrt{p}}{c}\right) \hphi\left(\frac{\log p}{\log R}\right) \frac{2\log p}{\sqrt{p}\log R}. \eea Applying the Petersson formula (Lemmas \ref{lem:ils21petersson} and \ref{lem:Peterssonjustnewforms}) to $S_1(\phi)$ yields \be S_1(\phi) \ = \  \sum_{c \equiv 0 \bmod N} \frac{Q_k^\ast(1;c)}{c} + O\left(\frac{R^{\sigma/2}}{N^{1-\gep}}\right). \ee This is very similar to the sum $\mathcal{P}_k^\ast(\phi)$ in equation (5.15) of \cite{ILS}, with $X=Y=1$, $L=1$, $M=N$. The difference is that (5.15) has an extra factor of $(k-1)N/12$, which is basically the cardinality of $H_k^\ast(N)$. We can use the results from sections 5 through 7 of \cite{ILS} to bound $Q_k^\ast(1;c)$. We have (see (7.1) of \cite{ILS}) that \be Q_k^\ast(m;c) \ \ll \ \widetilde{\gamma}_k(z) m P^{1/2} (kN)^\gep (\log 2c)^{-2}, \ee where $R = k^2N$, $P = R^\sigma$, $z = 4\pi m\sqrt{P}/c$ and $\widetilde{\gamma}(z) = 2^{-k}$ if $3z \le k$ and $k^{-1/2}$ otherwise. Thus \bea S_1(\phi) & \ \ll \ & \sum_{c\equiv 0 \bmod N} \frac{(k^2N)^{\sigma/2} (kN)^\gep}{c (\log 2c)^2} + N^{\frac{\sigma}2-1+\gep} \ \ll \ N^{\frac{\sigma}2-1+\gep} \eea (write $c = c'N$), which is negligible so long as $\sigma < 2$.
\end{proof}

\begin{rek} We briefly comment on where we use GRH for Dirichlet $L$-functions. If $\chi$ is a character modulo $c$, then under GRH we have \be \sum_{p \le x} \chi(p) \log p \ = \ \delta_\chi x + O\left(x^{1/2} \log^2 cx\right), \ee where $\delta_\chi = 1$ if $\chi$ is the principal character and $0$ otherwise. In Section 6 of \cite{ILS} they expand the Kloosterman sum. Setting \be G_\chi(n) \ = \ \sum_{a \bmod c} \chi(a) e^{2\pi i an/c}, \ee we find
 \bea \sum_{p \le c \atop p \not\div c} S(m,np;c) & \ = \ & \frac1{\varphi(c)} \left( \sum_{\chi \bmod c} \chi(a) S(m,an;c) \right) \cdot \left(\sum_{p \le x} \overline{\chi}(p) \log p\right) \nonumber\\ &=& \frac1{\varphi(c)} \sum_{\chi \bmod c} G_\chi(m)G_\chi(n) \left(\delta_\chi x + O\left(x^{1/2} \log^2 cx\right)\right). \eea If we did not assume GRH, the error term above would have to be replaced with something significantly larger. This estimate is a key input in the bound for $Q_k^\ast(m;c)$. \end{rek}

\begin{lem} Assume GRH for $\zeta(s)$, all Dirichlet $L$-functions and all $L(s,f)$. Let $\supp(\hphi) \subset (-\sigma, \sigma)$ with $\sigma < 2$, $N=1$, and consider the 1-level density averaged over the weights (see Theorem \ref{thm:1ldnumbthk} for an explicit statement). As $K\to\infty$ the 1-level density agrees with the prediction from the Ratios Conjecture up to errors of size $O(K^{-(5-\sigma)/6+\gep} + K^{\sigma-2+\gep})$.
\end{lem}

\begin{proof} As the proof is similar to our previous results, we merely highlight the differences. Following \cite{ILS} (Sections 8 and 9), we average over the weights as follows. Let $h$ be a Schwartz function compactly supported on $(0,\infty)$. The weighted $1$-level density is \bea \mathcal{A}^\ast(K;\phi) & \ = \ & \frac1{A^\ast(K)} \sum_{k\equiv 0 \bmod 2} \frac{24}{k-1} h\left(\frac{k-1}{K}\right) \sum_{f\in H_k^\ast(1)} D_{1,H_k^\ast(1);k^2}(\phi), \eea where \bea A^\ast(K) \ = \ \sum_{k\equiv 0 \bmod 2} \frac{24}{k-1} h\left(\frac{k-1}{K}\right) \left|H_k^\ast(1)\right| \ = \ \widehat{h}(0) K + O(K^{2/3}). \eea

The only pieces whose errors cannot be trivially added arise from $S_1(\phi)$ and $S_2(\phi)$ for each $k$; we now discuss how to handle these weighted averages.\footnote{Actually, we need to be a little more careful. The problem is that the analytic conductors are no longer constant; if $\supp(h) \subset (a,b)$ then the conductors basically run from $(aK)^2$ to $(bK)^2)$. Fortunately, an analysis of our previous arguments show that we do not need to localize the conductor exactly, but instead only up to a constant (see also equations (4.29) and (4.30) in \cite{ILS}, and the comments immediately after). Thus we may set $R=K^2$. The varying conductors here are significantly easier to handle than in other families, such as one-parameter families of elliptic curves \cite{Mil1}.} The main idea is to exploit the oscillation in the Bessel functions as $k$ varies. The argument is easier than that in \cite{ILS} due to the presence of the harmonic weights, though a similar result holds if we remove the weights (see \S\ref{sec:extendingsuppunweightedS1}).

We first handle the average of $S_1(\phi)$. Averaging over $k$ allows us to exploit the oscillation in the Bessel functions; this is the reason we are able to double the support. The main input is their Corollary 8.2, which says \bea I(x) \ = \ \sum_{k \equiv 0 \bmod 2} 2i^k h\left(\frac{k-1}{K}\right) J_{k-1}(x) \ \ll \ x K^{-4}, \eea where $x = 4\pi m \sqrt{p}/c$, $P = R^\sigma = K^{2\sigma}$, and for us $m=1$ (as \cite{ILS} remove the harmonic weights, they have a sum over $m \le Y$). Corollary 8.2 requires $x \ll K^{2-\gep}$, i.e., $\sigma < 2-\gep$. The analysis of the average of $S_1(\phi)$ is completed by feeding in the estimate from their equation (8.11), which yields a bound of $K^{\sigma+\gep-2}$ (remember we already executed the summation over $k$ when we bounded $I(x)$). 
Thus the total error from the sum over $k$ of the $S_1(\phi)$ terms is $O(K^{\sigma+\gep-2})$.

We now consider the average of $S_2(\phi)$. There are two major differences between this term and $S_1(\phi)$. The first is that the Kloosterman sums are $S(1,p^2;c)$ instead of $S(1,p;c)$. The second is that we have $\hphi\left(\frac{2\log p}{\log R}\right) \frac{\log p}{p\log R}$ instead of $\hphi\left(\frac{\log p}{\log R}\right) \frac{\log p}{\sqrt{p}\log R}$; this leads to a shorter prime sum of smaller terms. We can modify the arguments in Section 9 of \cite{ILS} (remembering, as in Lemma \ref{lem:s1phiksum}, that our sum is simpler as $L=X=Y=m=M=1$). Performing the averaging over $k$ yields \be\label{eq:suminvolvingq21c} \sum_c \frac{\mathcal{Q}^{(2)}(1;c)}{c}, \ee where \bea \mathcal{Q}^{(2)}(1;c) & \ = \ & 2\pi \sum_{p \neq N} S(1,p^2;c) I\left(\frac{4\pi p}{c}\right) \hphi\left(\frac{2\log p}{\log R}\right) \frac{2\log p}{p\log R} \eea and $I(x)$ is the sum of Bessel functions (see their equation (8.7)). By their Corollary 8.2 we have \be I(x) \ = \ -\frac{K}{\sqrt{x}} {\rm Im}\left\{\overline{\zeta}_8 e^{ix} \hbar\left(\frac{K^2}{2x}\right)\right\} + O\left(\frac{x}{K^4}\right).  \ee

The error term yields to an insignificant contribution to $\sum_c \mathcal{Q}^{(2)}(1;c)/c$ (much less than in \cite{ILS}, due to the remarks above). Trivially estimating the Kloosterman sum by $c^{1/2+\gep}$ and recalling $R = K^2$ yields a contribution of \be \sum_c \frac{1}{c} \sum_{p \le R^{\sigma/2}} c^{\foh+\gep}  \frac{p}{cK^4} \frac{1}{p} \ \ll \ K^{\sigma-4}, \ee which is negligible for $\sigma < 4$ (and smaller than $O(K^{-1/2})$ for $\sigma < 3.5$).

We now study the main term of $\mathcal{Q}^{(2)}(1;c)$. Following \cite{ILS} it is \be \mathcal{Q}^{(2)}(1;c) \ = \ -\frac{2K\sqrt{\pi c}}{\log R} \mathcal{T}(1;c), \ee where \bea \mathcal{T}(1;c) & \ = \ & \sum_{p\neq N} S(1,p^2;c) {\rm Im}\left\{\overline{\zeta}_8 \exp\left(\frac{4\pi i p}{c}\right) \hbar\left(\frac{cK^2}{8\pi p}\right) \right\} \hphi\left(\frac{2\log p}{\log R}\right) \frac{\log p}{p^{3/2}} \nonumber\\ \hbar(v) & \ = \ & \int_0^\infty \frac{h(\sqrt{u})}{\sqrt{2\pi u}}\ e^{iuv} du. \eea We do not need as delicate an analysis as in \cite{ILS}. This is because of the extra $\sqrt{p}$ in the denominator and the fact that the prime sums are up to $R^{\sigma/2}$ and not $R^\sigma$. We trivially estimate the Kloosterman sums and use the bound on $\hbar$ from \cite{ILS}: for any $A>0$, $\hbar(v) \ll v^{-A}$. Taking $A = 1 + \delta$ yields \be \mathcal{T}(1;c) \ \ll \ \sum_{p \le R^{\sigma/2}} \frac{c^{\foh+\gep} p^{1+\delta}}{c^{1+\delta} K^{2+2\delta}} \frac{\log p}{p^{3/2}} \ \ll \ \frac{K^{\frac{\sigma}2+\sigma \delta - 2 - 2\delta}}{c^{\foh+\delta-\gep}}. \ee We substitute this into \eqref{eq:suminvolvingq21c}, and find a contribution bounded by \be \sum_c  \frac{K \sqrt{c}}{c} \frac{K^{\frac{\sigma}2+\sigma \delta - 2 - 2\delta}}{c^{\foh+\delta-\gep}} \ \ll \ K^{-(\foh+\delta)\left(2-\sigma\right)}. \ee By taking $\delta$ sufficiently large, we can make this sum as small as we desire (and thus smaller than the contribution from the averaged $S_1(\phi)$).
\end{proof}

\begin{rek}\label{rek:errorILS810} There is a mistake right before equation (8.10) in \cite{ILS}; it should read \be 4I'(x)\ =\ \frac2{K} h'\left(\frac{x+\eta}{K}\right) + O\left(\frac{x}{K^2}\right), \ \ \eta \in (-1,1); \ee fortunately all \cite{ILS} use in their argument is that $I'(x) \ll K^{-1}$ when $x \ll K^{2-\gep}$, and that is true. Also, it is worth noting that our analysis of $\mathcal{Q}^{(2)}$ uses their results for the family $\{{\rm sym}^2 f: f \in H_k^\ast(N)\}$; our support is significantly larger because (1) this is now a $1/p$ term and not a $1/\sqrt{p}$; (2) we sum over $p \le R^{\sigma/2}$ and not $p \le R$.  \end{rek}

\subsection{Hypothesis $S$ and further extensions}

Iwaniec, Luo and Sarnak \cite{ILS} show how a hypothesis on the size of some classical exponential sums over the primes can be used to increase the support to beyond $(-2, 2)$. They consider \\

\noindent \emph{Hypothesis S: For any $x \ge 1$, $c \ge 1$ and $a$ with $(a,c) = 1$ we have} \be\label{eq:hypothesisS} \sum_{p \le x \atop p \equiv a \bmod c} \exp\left(\frac{4\pi i \sqrt{p}}{c}\right) \ \ll_\gep c^A x^{\alpha + \gep}, \ee \emph{where $\alpha, A$ are constants with $A \ge 0$, $1/2 \le \alpha \le 3/4$ and $\gep$ is any positive number.} \ \\

They present numerous arguments (see their Section 10 and their Appendix C) in support of the belief that Hypothesis S holds with $A=0$ and $\alpha = 1/2$; however, any $\alpha < 3/4$ suffices to increase the support past $(-2, 2)$.\footnote{Vinogradov proved Hypothesis S with $\alpha = 7/8$; assuming the standard density hypothesis for Dirichlet $L$-functions allows one to take $\alpha = 3/4$.} We show how this hypothesis allows us to extend our computations. As \cite{ILS} were only concerned with the main term, their error bounds are too crude; however, some additional book-keeping suffices to obtain all lower order terms up to a power savings in the family's cardinality. 

To prove the third statement in Theorem \ref{thm:1ldnumbthk} we need to study the weighted averages over $k$ of $S_i(\phi)$ $(i \in \{1,2,3\}$). We note that they use the Petersson weights in their Section 10 (and thus we are using the same normalization for our sums). From Lemma \ref{lem:S2S3Nk}, we see may average $S_3(\phi)$ and obtain a contribution bounded by $O(K^{-(5-\sigma)/6+\gep})$. The analysis in Section 10 of \cite{ILS} handles $S_1(\phi)$, and shows (under the assumption that Hypothesis S holds) that it is $O(K^{-2(2.5-\sigma)} + K^{-(2A+11/2+\gep)\left(1 - \frac{2\alpha+A+5/4}{2A+11/2+\gep}\right)})$. In particular, taking $A=0$ and $\alpha=1/2$ yields the weighted average of $S_1(\phi)$ is $O(K^{-2(2.5-\sigma)}+K^{-\frac{11}{2}(1-\frac{9}{22}\sigma)})$.

We are left with bounding the weighted average over $k$ of $S_2(\phi)$, remembering $R = K^2$. In \cite{ILS} it is shown to be $O\left(\frac{\log\log K}{\log K}\right)$, which does not suffice for our purposes. This term contributes \bea \frac1{B(K)} \sum_{p \le R^{\sigma/2}} \mathcal{B}(p^2,1) \hphi\left(\frac{2\log p}{\log R}\right) \frac{\log p}{p\log R}, \eea where $B(K) = \widehat{h}(0)K + O(1)$ (with $\widehat{h}(0) \neq 0)$, \bea \mathcal{B}(p^2,1) \ = \ -\frac{\sqrt{\pi}K}{\sqrt{p}}{\rm Im}\left\{\overline{\zeta}_8 \sum_c \frac{S(1,p^2,c)}{c}\ e^{\frac{4\pi i p}{c}} \hbar\left(\frac{cK^2}{8\pi p}\right) \right\} + O\left(pK^{-4}\right) \eea and $\hbar(v) \ll v^{-\delta}$ for any $\delta > 0$. The $O(pK^{-4})$ term in $\mathcal{B}(p^2,1)$ leads to a contribution of size $K^{-(5-\sigma)}$, which is dwarfed by the other error terms. We trivially bound the main term in $\mathcal{B}(p^2,1)$ by using $S(1,p^2,c) \ll c^{\foh+\gep}$ and $\hbar(cK^2/8\pi p) \ll p^\delta / (cK^2)^\delta$ for some $\delta > 1/2$ (we take $\delta > 1/2$ so that the resulting $c$-sum converges). This yields a contribution to the average of $S_2(\phi)$ of \bea \frac1{K} \sum_{p \le R^{\sigma/2}} \frac{K}{\sqrt{p}} \sum_c \frac{c^{\foh+\gep}}{c} \frac{p^\delta}{c^\delta K^{2\delta}} \frac1{p} \ \ll \ K^{-2\delta} \sum_p p^{\delta-\foh-1} \ \ll \ K^{(\delta-\foh)\sigma - 2\delta}. \eea Taking $\delta$ just a little larger than $1/2$ shows that this error is also dwarfed by our existing errors (as well as being $O(N^{-1+\gep})$, which completes the proof.

%%%%%%%%%%%%%%%%%%%%%%%%%%%%%%%%%%%%%%%%%%%%%%%%%%%%%%%%%%%%%%%%%%%%%%%%%%%%%%%%%%%%%%%%%%%%%
%%%%%%%%%%%%%%%%%%%%%%%%%%%%%%%%%%%%%%%%%%%%%%%%%%%%%%%%%%%%%%%%%%%%%%%%%%%%%%%%%%%%%%%%%%%%%

\section{Calculating the unweighted 1-level density}\label{sec:extendingsuppunweightedS1}

Much effort was spent removing the harmonic weights in \cite{ILS}. Below we remove them for our family and calculate the lower order terms. We see some new, lower order terms which did not appear in either the expansion from the Ratios Conjecture or our number theory computations. This is not entirely surprising, as those computations were for weighted sums.

We prove Theorem \ref{thm:1ldnumbthunweighted}. We first concentrate on the unweighted version of $S_1(\phi)$, which yields negligible contributions for $\sigma < 2$. We then analyze the unweighted versions of $S_2(\phi)$ and $S_3(\phi)$, and find new lower order terms. The other terms in \eqref{eq:NTexpansion1ldHknast} are unaffected by removing the weights. We conclude by determining the prediction from the Ratios Conjecture for the unweighted 1-level density, and show agreement with number theory.

\subsection{Analyzing the unweighted $S_1(\phi)$}

Below we modify the arguments in \cite{ILS} to show that $S_{1,{\rm unwt}}(\phi)$ has negligible contribution for $\sigma < 2$ when we do not include the harmonic weights.

\begin{lem}\label{lem:S1phiNinfty} Assume GRH for $L(s,f)$. If $\supp(\hphi) \subset (-\sigma, \sigma)$ with $\sigma < 2$, then $S_{1,{\rm unwt}}(\phi) \ll N^{-(2-\sigma)/6 + \gep}$ as $N\to\infty$ through the primes, where $S_{1,{\rm unwt}}(\phi)$ is defined analogously as $S_1(\phi)$ except now we do not include the harmonic weights. \end{lem}

\begin{proof} We use the expansions in \cite{ILS} for $\Delta_{k,N}^\ast(p)$, remembering to divide by $|H_k^\ast(N)|$. Let $X$ and $Y$ be two arbitrary parameters (depending on $N$) to be determined later. We let $\gep$ denote an arbitrarily small number (not necessarily the same value from line to line). We write \be \Delta_{k,N}^\ast(p) \ = \ \Delta_{k,N}'(p) + \Delta_{k,N}^\infty(p),\ee where ((2.63) of \cite{ILS}) \bea \Delta_{k,N}'(p) & \ = \ & \frac{k-1}{12} \sum_{LM=N \atop L \le X} \frac{\mu(L)M}{\nu((n,L))} \sum_{(m,M) = 1 \atop m \le Y} \frac{\Delta_{k,M}(m^2,n)}{m} \eea and $\Delta_{k,N}^\infty(p)$ is the complementary sum. Here \be \nu(\ell) \ = \ \left[\Gamma_0(1): \Gamma_0(\ell)\right] \ = \ \ell \prod_{p|\ell} \frac{p+1}{p}. \ee As $N$ is prime, so long as $X < N$ then in $\Delta_{k,N}'(p)$ the only term is when $L=1$ and $M=N$. Thus \bea S_{1,{\rm unwt}}(\phi)  & \ = \ & \frac1{|H_k^\ast(N)|}\sum_{p \notdiv N} \Delta_{k,N}'(p) \hphi\left(\frac{\log p}{\log R}\right) \frac{2\log p}{\sqrt{p}\log R}\nonumber\\ & & \ \ +\ \frac1{|H_k^\ast(N)|} \sum_{p \notdiv N} \Delta_{k,N}^\infty(p) \hphi\left(\frac{\log p}{\log R}\right) \frac{2\log p}{\sqrt{p}\log R} \nonumber\\ & = & S_{1,{\rm unwt}}'(\phi) + S_{1,{\rm unwt}}^\infty(\phi). \eea

We first show there is no contribution from the complementary sum. As we are going for a power savings in $N$ and not just attempting to understand the main term, we choose different values for $X$ and $Y$ then in \cite{ILS}, and argue slightly differently. Assuming the Riemann hypothesis for $L(s,f)$, if $\log Q \ll \log kN$ then (Lemma 2.12 of \cite{ILS}) \be  \sum_{(p,N) = 1 \atop p \le Q} \Delta_{k,N}^\infty(p) \frac{\log p}{\sqrt{p}} \ \ll \ kN (pkNXY)^\gep (X^{-1} + Y^{-1/2}). \ee Using partial summation,  the compact support of $\hphi$ and $H_k^\ast(N) \ll kN$ shows that the complementary sum piece is bounded by \bea S_{1,{\rm unwt}}^\infty(\phi) & \ = \ & \frac{1}{kN\log R} \int^{R^\sigma} kN(pkNXY)^\gep (X^{-1}+Y^{-1/2}) \left|\hphi'\left(\frac{\log p}{\log R}\right)\right|\frac{dp}{p\log R}\nonumber\\ &  \ \ll \ & N^\gep(X^{-1}+Y^{-1/2}). \eea

We now analyze the contribution from $\Delta_{k,N}'(p)$. The formulas from \cite{ILS} simplify greatly as we only have one $(L,M)$ pair, and as $p$ is not a perfect square there are no main terms. We have \bea S_{1,{\rm unwt}}'(\phi) & \ = \ & \frac{(k-1)N}{12 |H_k^\ast(N)|} \sum_{(m,N)=1 \atop m \le Y} \frac1{m} \sum_{c\equiv 0 \bmod N} \frac{Q_k^\ast(m;c)}{c}, \eea where \bea Q_k^\ast(m;c) \ = \ 2\pi i^k \sum_{p\notdiv N} \frac{S(m^2,p;c)}{c} J_{k-1}\left(\frac{4\pi m\sqrt{p}}{c}\right) \hphi\left(\frac{\log p}{\log R}\right) \frac{2\log p}{\sqrt{p}\log R}. \eea In (5.14) of \cite{ILS} they set $X = Y = (kN)^\gep$; however, their estimates of $Q_k^\ast(m;c)$ are independent of $X$ and $Y$, and we may thus use their results. We have (see (7.1) of \cite{ILS}) that \be Q_k^\ast(m;c) \ \ll \ \widetilde{\gamma}_k(z) m P^{1/2} (kN)^\gep (\log 2c)^{-2}, \ee where $R = k^2N$, $P = R^\sigma$, $z = 4\pi m\sqrt{P}/c$ and $\widetilde{\gamma}(z) = 2^{-k}$ if $3z \le k$ and $k^{-1/2}$ otherwise. Thus \bea S_{1,{\rm unwt}}'(\phi) & \ \ll \ & \sum_{(m,N)=1 \atop m \le Y} 1 \sum_{c\equiv 0 \bmod N} \frac{(k^2N)^{\sigma/2} (kN)^\gep}{c (\log 2c)^2} \ \ll \ N^{\frac{\sigma}2-1+\gep} Y. \eea

Combining our estimates yields \be S_{1,{\rm unwt}}(\phi) \ \ll \ N^{\frac{\sigma}2-1+\gep} Y + N^\gep (X^{-1} + Y^{-1/2}). \ee We may take $X=N-1$ (as $N$ is prime). Equalizing the two errors involving $Y$, we find we should take $Y = N^{(2-\sigma)/3}$, which gives $S_{1,{\rm unwt}}(\phi) \ll N^{(2-\sigma)/6}$.
\end{proof}

\begin{lem} Assume GRH for $L(s,f)$. If $\supp(\hphi) \subset (-\sigma, \sigma)$ with $\sigma < 2$, then $S_{1,{\rm unwt}}(\phi) \ll K^{-(2-\sigma)/6 + \gep}$ as $N=1$ and $K\to\infty$ (where we average over the weights). \end{lem}

\begin{proof} As the proof is similar to Lemma \ref{lem:S1phiNinfty}, we merely highlight the differences. Following \cite{ILS} (Section 8), we average over the weights as follows. Let $h$ be a Schwartz function compactly supported on $(0,\infty)$. We consider the weighted $1$-level density \bea \mathcal{A}^\ast(K;\phi) & \ = \ & \frac1{A^\ast(K)} \sum_{k\equiv 0 \bmod 2} \frac{24}{k-1} h\left(\frac{k-1}{K}\right) \sum_{f\in H_k^\ast(1)} D_{1,H_k^\ast(1);k^2}(\phi), \eea where \bea A^\ast(K) \ = \ \sum_{k\equiv 0 \bmod 2} \frac{24}{k-1} h\left(\frac{k-1}{K}\right) \left|H_k^\ast(1)\right| \ = \ \widehat{h}(0) K + O(K^{2/3}). \eea

The pieces whose errors cannot be trivially added arises from $S_{i,{\rm unwt}}(\phi)$ ($i \in \{1,2,3\}$) for each $k$. We analyze the weighted average of $S_1(\phi)$ below, and then study the other two in \S\ref{sec:analyzingunwts2s3phi}. The main idea is to exploit the oscillation in the Bessel functions as $k$ varies.

In their Lemma 2.12 we now take $X=1$ and $Y = K^\delta$. The complementary sum gives an error bounded by $k^\gep Y^{-1/2}$. The averaging over $k$ allows us to exploit the oscillation in the Bessel functions; this is the reason we are able to double the support. The main input is their Corollary 8.2, which says \bea I(x) \ = \ \sum_{k \equiv 0 \bmod 2} 2i^k h\left(\frac{k-1}{K}\right) J_{k-1}(x) \ \ll \ x K^{-4}, \eea where $x = 4\pi m \sqrt{P}/c$ and $P = R^\sigma = K^{2\sigma}$. Corollary 8.2 requires $x \ll K^{2-\gep}$. In their arguments they take $Y = K^\gep$, and thus for them $m\le K^\gep$ (recall $m \le Y$). As we are interested in sharper error estimates, we must take $Y$ a small power of $K$. This leads to a slight reduction in the support (our condition on $x$ forces $\sigma < 2 - \delta$). The proof is completed by feeding in the estimate from their equation (8.11), which yields a bound of $K^{\sigma+\delta+\gep-2}$ for the term from the non-complementary piece (remember we already executed the summation over $k$ when we bounded $I(x)$).

Thus the total error from the sum over $k$ of the $S_{1,{\rm unwt}}(\phi)$ terms is $O(K^\gep Y^{-1/2} + K^{\sigma+\delta+\gep-2})$. Equalizing the errors yields $\delta = (2-\sigma)/3$, or the total error from the weighted $S_{1,{\rm unwt}}(\phi)$ terms is $O(K^{-(2-\sigma)/6})$.
\end{proof}

\begin{rek} There is a mistake right before equation (8.10) in \cite{ILS}; see Remark \ref{rek:errorILS810}. \end{rek}

\subsection{Analyzing the unweighted $S_2(\phi)$ and $S_3(\phi)$}\label{sec:analyzingunwts2s3phi}

We now modify our investigation of $S_2(\phi)$ and $S_3(\phi)$ and remove the weights. We set \bea S_{3, {\rm unwt}}(\phi) & \ = \ & \frac{2}{|H_k^\ast(N)|} \sum_{f \in
H_k^\ast(N)} \sum_{p\neq N} \sum_{\nu = 3}^\infty
\frac{\lambda_f(p^\nu) - \lambda_f(p^{\nu-2})}{p^{\nu/2}}\
\hphi\left(\nu\frac{\log p}{\log R}\right) \frac{\log p}{\log
R}\nonumber\\  S_{2, {\rm unwt}}(\phi) & \ = \ & \frac{2}{|H_k^\ast(N)|} \sum_{f \in H_k^\ast(N)} \sum_{p\neq N} \frac{\gl_f(p^2)}{p} \hphi\left(2\frac{\log p}{\log R}\right) \frac{\log p}{\log R}. \eea

We argue as in the analysis of $S_{1, {\rm unwt}}(\phi)$. As the two terms are handled analogously, we concentrate on $S_{3, {\rm unwt}}(\phi)$. The analysis is significantly easier than the analysis of $S_{1, {\rm unwt}}(\phi)$ due to the higher power of primes (both in dividing by larger quantities and restricting further the summation over primes). Let $v = \nu$ or $\nu-2$. We must study the pure sums \be \Delta_{k,N}^\ast(p^v) \ = \ \sum_{f\in H_k^\ast(N)} \gl_f(p^v). \ee From Proposition 2.13 of \cite{ILS} we have \be\label{eq:unweightedDeltaknastn} \Delta_{k,N}^\ast(n) \ = \ \frac{(k-1)\varphi(N)}{12\sqrt{n}}\ \delta_{n,\Box} + O\left(\frac{(kN)^{2/3} n^{1/6}}{\sqrt{(n,N)}}\right) \ee where the main term is present only if $n$ is a square and $(n,N) = 1$. The contribution from the error term to $S_{3, {\rm unwt}}(\phi)$ is bounded by \be \sum_{\nu \le \log_2 R} \sum_{p \le R^{\sigma/3}} \frac{p^{3/6}}{p^{3/2}} \frac{(kN)^{2/3}}{kN} \ \ll \ \frac{\log^2 R}{(kN)^{1/3}}. \ee Thus the error term yields a negligible contribution.

The main term from Proposition 2.13, however, is a different story. Whenever $\nu$ is even it \emph{will} contribute, and yields \be\label{eq:lowerordertermnotseeninratiosconj} \sum_{\nu \equiv 0 \bmod 2 \atop \nu \ge 4} \sum_{p \neq N} \frac{1-p}{p^{\nu}}\ \hphi\left(\nu\frac{\log p}{\log R}\right) \frac{\log p}{\log R}. \ee
The unweighted $S_2(\phi)$ term will also contribute, as it involves $\gl_f(p^2)$. It gives another secondary term of size $1/\log R$, as well as an error of size $O(N^{-(6-\sigma)/6+\gep})$. Substituting everything into \eqref{eq:NTexpansion1ldHknast} yields \bea D_{1,H_k^\ast(N);R}(\phi) & \ = \ & \frac1{\log R}
\int_{-\infty}^\infty \left(2 \log \frac{\sqrt{N}}{\pi} +
\psi\left(\frac14 + \frac{k\pm 1}4 +\frac{2\pi i t}{\log R}\right)
\right) \phi(t) dt \nonumber\\ & & \ \ \  +2 \sum_{\nu \equiv 0 \bmod 2 \atop \nu \ge 2} \sum_{p \neq N} \frac{p-1}{p^\nu} \hphi\left(\nu \frac{\log p}{\log R}\right)\frac{\log p}{\log R}\nonumber\\ & & \ \ \   + O\left(N^{-1/2} + N^{-(2-\sigma)/6+\gep}\right);\eea the sum starts at $\nu = 2$ and not $\nu = 4$ as we have incorporated both $S_{2, {\rm unwt}}(\phi)$ and the $\sum_p 1/p$ term in \eqref{eq:NTexpansion1ldHknast}. This completes the analysis of the number theory terms in Theorem \ref{thm:1ldnumbthunweighted}.

\subsection{Unweighted Ratios Prediction}

We sketch the derivation of the prediction for the unweighted 1-level density from the Ratios Conjecture, which completes the proof of Theorem \ref{thm:1ldnumbthunweighted}. We concentrate on the case $N\to\infty$ through the primes. As the analysis is similar to the weighted case, we just highlight the new terms.

The Ratios Conjecture recipe states we should replace averages over the family by the main term, throwing away the `small' error. The problem is that while $\sum_{f \in H_k^\ast(N)} \omega_f^\ast(N) \lambda_f(n)$ is small for $n \ge 2$, it is \emph{not} small for $n$ a perfect square if we drop the weights (see \eqref{eq:unweightedDeltaknastn}).

We highlight the changes to Theorem \ref{thm:ratiosconj1level1} from studying the unweighted family $H_k^\ast(N)$. The first change is in Lemma \ref{lem:ratiosconjRalphagamma}. Originally we had the first term of $R_{H_k^\ast(N)}(\alpha,\gamma)$ was \be \prod_{p} \left(1 -
\frac{1}{p^{1+\alpha+\gamma}} + \frac1{p^{1+2\gamma}}\right); \ee now, however, we shall see it is \be \prod_p \left(1 - \frac{p+1}{p} \frac{1}{p^{1+\alpha+\gamma}} + \frac1{p^{1+2\gamma}}\right) \cdot \left(1 - \frac1{p^{2+2\alpha}}\right)^{-1}. \ee We do not worry about the changes to the second term, as it leads to a contribution of size $O(1/N)$.

The proof follows from mirroring the calculation in Lemma \ref{lem:ratiosconjRalphagamma}. We again assume we may split the sum into a product over primes. We constantly use (from \eqref{eq:standardformslambdafp}) $\lambda_f(p) \lambda_f(p^\nu) = \lambda_f(p^{\nu-1}) + \lambda_f(p^{\nu+1})$. We average over the family by using \eqref{eq:unweightedDeltaknastn}; which says there is no main term unless we are evaluating at a square. Thus below we drop all terms involving $\gl_f(p)\gl_f(p^{2k})$ or $\gl_f(p^{2k+1})$, as these yield lower order terms. We also ignore the product over $p \ge x$, as those terms vanish when we complete the product by sending $x\to\infty$. Thus we have   \bea   \sum_{m \le x \atop h} \frac{\mu_f(h) \lambda_f(m)}{h^{\foh+\gamma} m^{\foh+\alpha}}= \prod_{p \le x} \left(1 - \frac{\lambda_f(p)}{p^{\foh+\gamma}} + \frac{1}{p^{1+2\gamma}} \right) \cdot \left(1 + \frac{\gl_f(p)}{p^{\foh+\alpha}} + \frac{\gl_f(p^2)}{p^{1+2\alpha}} + \cdots \right)\ \ \eea contributes \bea \prod_{p \le x} \left[ \left(1+\frac1{p^{1+2\gamma}}\right) \sum_{k=0}^\infty \frac{\gl_f(p^{2k})}{(p^{1+2\alpha})^k} - \frac1{p^{1+\ga+\gamma}} \sum_{k=0}^\infty \frac{\lambda_f(p^{2k}) + \lambda_f(p^{2k+2})}{(p^{1+2\ga})^k} \right]. \eea
We now average over the family and divide by the family's cardinality; this replaces $\gl_f(p^{2\ell})$ with $1/p^\ell$ (remember we ignore all error terms). Using the geometric series formula and completing the product, after some simple algebra we find a contribution of \be \prod_p \left(1  + \frac1{p^{1+2\gamma}} - \frac{p+1}{p} \frac1{p^{1+\ga+\gamma}} \right) \left(1 + \frac1{p^{2+2\alpha}}\right)^{-1}. \ee

For the Ratios Conjecture prediction, however, we need the derivative of this piece with respect to $\alpha$ when $\alpha = \gamma = r$. We must therefore modify Lemma \ref{lem:Rprimealphagammar} as well. This piece contributes to $R_{H_k^\ast(N)}'(r,r)$ a factor of \be \sum_p \frac{(p-1)\log p}{p^{2+2r}-1} \ = \ \sum_p (p-1)\log p \sum_{k=1}^\infty \frac{1}{p^{2k} p^{2rk}}; \ee this is very similar to what we previously had for $R_{H_k^\ast(N)}'(r,r)$, namely \be \sum_p \frac{\log p}{p^{1+2r}}. \ee

This change propagates to Lemma \ref{lem:ratiosconj1level1}, where instead of \be \int_{-\infty}^\infty g(t) \sum_p \frac{\log p}{p^{1+2it}}dt \ = \ \sum_p \hg\left(\frac{2\log p}{2\pi}\right) \ee we now have   \bea\ \int_{-\infty}^\infty g(t) \sum_p \sum_{k=1}^\infty \frac{(p-1)\log p }{p^{2k} p^{2itk}}dt  &\ = \ & \sum_p \sum_{k=1}^\infty \frac{(p-1)\log p}{p^{2k}} \int_{-\infty}^\infty g(t) e^{-2\pi i \frac{2kt \log p}{2\pi}} dt \nonumber\\ &\ = \ & \sum_p \sum_{k=1}^\infty \frac{(p-1)\log p}{p^{2k}} \hg\left(\frac{2k\log p}{2\pi}\right). \eea

Setting $g(t) = \phi\left(\frac{t\log R}{2\pi}\right)$ and collecting all the terms completes the proof of the Ratios Conjecture's prediction in Theorem \ref{thm:1ldnumbthunweighted}.

\appendix

%%%%%%%%%%%%%%%%%%%%%%%%%%%%%%%%%%%%%%%%%%%%%%%%%%%%%%%%%%%%%%%%%%%%%%%%%%%%%%%%%%%%%%%%%%%%%
%%%%%%%%%%%%%%%%%%%%%%%%%%%%%%%%%%%%%%%%%%%%%%%%%%%%%%%%%%%%%%%%%%%%%%%%%%%%%%%%%%%%%%%%%%%%%

\section{Petersson Formula}\label{sec:PeterssonFormula}

Below we record several useful variants of the Petersson formula. We define \be \Delta_{k,N}(m,n) \ = \ \sum_{f \in \mathcal{B}_k(N)} \omega_f(N) \lambda_f(m) \lambda_f(n). \ee We quote the following versions of the Petersson formula from \cite{ILS} (to match notations, note that $\sqrt{\omega_f(N)} \lambda_f(n) = \psi_f(n)$).

\begin{lem}[\cite{ILS}, Proposition 2.1]\label{lem:ils21petersson} We have \be \Delta_{k,N}(m,n) \ = \ \delta(m,n) + 2\pi i^k \sum_{c\equiv 0 \bmod N} \frac{S(m,n;c)}{c} J_{k-1}\left(\frac{4\pi\sqrt{mn}}{c}\right), \ee where $\delta(m,n)$ is the Kronecker symbol, \be S(m,n;c)\ =\ \sideset{}{^*}\sum_{d \bmod c} \exp\left(2\pi i\frac{md+n\overline{d}}{c}\right)\ee is the classical Kloosterman sum ($d\overline{d} \equiv 1 \bmod c$), and $J_{k-1}(x)$ is a Bessel function.
\end{lem}

We expect the main term to arise only in the case when $m=n$ (though as shown in \cite{HM,ILS}, the non-diagonal terms require a sophisticated analysis for test functions with sufficiently large support). We have the following estimates.

\begin{lem}[\cite{ILS}, Corollary 2.2]\label{lem:ilscor22} We have \be \Delta_{k,N}(m,n) \ = \ \delta(m,n) + O\left(\frac{\tau(N)}{Nk^{5/6}} \ \frac{(m,n,N) \tau_3((m,n))}{\sqrt{(m,N)+(n,N)}} \ \left(\frac{mn}{\sqrt{mn} + kN}\right)^{1/2} \log 2mn\right), \ee where $\tau_3(\ell)$ denotes the corresponding divisor function. \end{lem}

We can significantly decrease the error term if $m$ and $n$ are small relative to $kN$.

\begin{lem}[\cite{ILS}, Corollary 2.3]\label{lem:ils23} If $12\pi\sqrt{mn} \le kN$ we have \be \Delta_{k,N}(m,n) \ = \ \delta(m,n) + O\left(\frac{\tau(N)}{2^k N^{3/2}} \ \frac{(m,n,N) \sqrt{mn}}{\sqrt{(m,N)+(n,N)}}\ \tau((m,n))\right). \ee
\end{lem}

In this paper we consider two cases, $N=1$ and $k\to\infty$ or $k$ fixed and $N \to \infty$ through prime values. In the first case, there is no problem with using the above formulas; however, in the second case we must be careful. $\Delta_{k,N}(m,n)$ is defined as a sum over all cusp forms of weight $k$ and level $N$; in practice we often study the families $H_k^\sigma(N)$ of cuspidal newforms of weight $k$ and level $N$ (if $\sigma = +$ we mean the subset with even functional equation, if $\sigma = -$ we mean the subset with odd functional equation, and if $\sigma = \ast$ we mean all). Thus we should remove the contribution from the oldforms in our Petersson expansions. Fortunately this is quite easy if $N$ is prime, as then the only oldforms are those of level 1 (following \cite{ILS}, with additional work we can readily handle $N$ square-free). We have (see (1.16) of \cite{ILS}) \be |H_k^\pm(N)| \ \sim \ \frac{k-1}{24} \ \varphi(N), \ee where $\varphi(N)$ is Euler's totient function (and thus equals $N-1$ for $N$ prime). The number of cusp forms of weight $k$ and level $1$ is (see (1.15) of \cite{ILS}) approximately $k/12$. As $\lambda_f(n) \ll \tau(n) \ll n^\gep$ and $\omega_f^\ast(N) \ll N^{-1+\gep}$, we immediately deduce

\begin{lem}\label{lem:Peterssonjustnewforms} Let $\mathcal{B}^{\rm new}_k(N)$ be a basis for $H_k^\ast(N)$ and let $\omega_f^\ast(N)$ be as in \eqref{eq:omegaastfN}. For $N$ prime, we have \bea \sum_{f \in \mathcal{B}^{\rm new}_k(N)} \omega_f^\ast(N) \lambda_f(m)\lambda_f(n) & \ = \ & \Delta_{k,N}(m,n) + O\left(\frac{(mnN)^\gep k}{N}\right). \eea Substituting yields \bea & & \sum_{f \in \mathcal{B}^{\rm new}_k(N)} \omega_f^\ast(N) \lambda_f(m)\lambda_f(n) \ = \ \delta(m,n) + O\left(\frac{(mnN)^\gep k}{N}\right)  \nonumber\\ & & \ \ \ \ \ \ + \ O\left(\frac{\tau(N)}{Nk^{5/6}} \ \frac{(m,n,N) \tau_3((m,n))}{\sqrt{(m,N)+(n,N)}} \ \left(\frac{mn}{\sqrt{mn} + kN}\right)^{1/2} \log 2mn\right),\ \ \ \ \ \eea while if $12\pi\sqrt{mn} \le kN$ we have \bea & & \sum_{f \in \mathcal{B}^{\rm new}_k(N)} \omega_f^\ast(N) \lambda_f(m)\lambda_f(n) \ = \ \delta(m,n)  \nonumber\\ & & \ \ \ \ \ \ + \ O\left(\frac{\tau(N)}{2^k N^{3/2}} \ \frac{(m,n,N) \sqrt{mn}}{\sqrt{(m,N)+(n,N)}}\ \tau((m,n))\right) + O\left(\frac{(mnN)^\gep k}{N}\right). \ \ \ \ \eea
\end{lem}

\begin{proof} The proof follows by using equations \eqref{eq:omegaastfN} and \eqref{eq:omegaastfNb} in the Petersson lemmas. \end{proof}

%%%%%%%%%%%%%%%%%%%%%%%%%%%%%%%%%%%%%%%%%%%%%%%%%%%%%%%%%%%%%%%%%%%%%%%%%%%%%%%%%%%%%%%%%%%%%
%%%%%%%%%%%%%%%%%%%%%%%%%%%%%%%%%%%%%%%%%%%%%%%%%%%%%%%%%%%%%%%%%%%%%%%%%%%%%%%%%%%%%%%%%%%%%

\section{Useful estimates}\label{sec:usefulestimates}

\begin{lem}\label{lem:xlgammafactors} Consider $\left|X_L\left(\frac{1+w}2+\frac{2\pi i t}{\log R}\right)\right|$. If $w=0$ it is $O(1)$, while if $w=\frac{2k-1}3$ it is $O\left(\left(\frac{2009}{\sqrt{N}}\right)^{\frac{2k-1}3} \cdot k^{-k/3}\right)$. \end{lem}

\begin{proof} We have \bea X_L\left(\frac{1+w}2+\frac{2\pi i t}{\log
R}\right) \ = \ \left(\frac{\sqrt{N}}{2\pi}\right)^{-w-\frac{4\pi i
t}{\log R}} \frac{\Gamma\left(\frac{1-w}4+\frac{k-1}2-\frac{\pi i
t}{\log R}\right)}{\Gamma\left(\frac{1+w}4+\frac{k-1}2+\frac{\pi i
t}{\log R}\right)}. \eea The claim follows from analyzing the
ratio of the Gamma factors. As $|\Gamma(x+iy)| =
|\Gamma(x-iy)|$ we may replace $-\pi i t/\log R$ with $+\pi i t /
\log R$ in the Gamma function in the numerator. The proof is thus trivial for $w=0$. If $w>0$ then we use the identity \be\label{eq:betaidentity} \frac{\Gamma(a+iy)\Gamma(b-a)}{\Gamma(b+iy)}
\ = \ \int_0^1 t^{a+iy-1}(1-t)^{b-a-1}dt. \ee Note
that if $a = b$ then our bound is poor due to the presence of the
$\Gamma(b-a)$ term; however, for us that would correspond to $w=0$,
and in that case the ratio of the Gamma factors is just $1$.

We apply \eqref{eq:betaidentity} with $a = \frac{1-w}4+\frac{k-1}2$ and $b-a = \frac{w}2$. We take $w = \frac{2k-1}3$ (chosen so that $b-a = a = \frac{2k-1}6$). We want $w < 2k-1$ as in our applications we will be shifting contours, and we want to avoid the pole of the numerator. The ratio of the Gamma factors, when $w=\frac{2k-1}3$, is \bea \left| \frac{\Gamma\left(\frac{2k-1}6+\frac{\pi i t}{\log R}\right)}{\Gamma\left(\frac{2k-1}3+\frac{\pi i t}{\log R}\right)}\right| \ \le \ \frac1{\Gamma\left(\frac{2k-1}6\right)} \int_0^1 \left(t \cdot (1-t)\right)^{\frac{2k-1}6-1} dt \ \le \ \frac{4\left(1/4\right)^{\frac{2k-1}6} }{\Gamma\left(\frac{2k-1}6\right)}  \eea (as the integrand is largest when $t=1/2$). Thus for $w=\frac{2k-1}6$, applying Stirling's formula to $\Gamma\left(\frac{2k-1}{6}\right)$ we find \be \left|X_L\left(\frac{k+1}3+\frac{2\pi i t}{\log R}\right)\right| \ \le \ \left(\frac{\pi}{\sqrt{N}}\right)^{\frac{2k-1}3} \frac4{\Gamma\left(\frac{2k-1}6\right)} \ \le \ \left(\frac{2009}{\sqrt{N}}\right)^{\frac{2k-1}3} \cdot k^{-k/3}. \ee \end{proof}

\begin{rek}\label{rek:applyingHolder} In the hope that this might be useful to other researchers on related problems, we sketch an alternative attack\footnote{We shall use H$\ddot{{\rm o}}$lder's inequality. See \cite{HM} for another application of H$\ddot{{\rm o}}$lder's inequality to bounding error terms in $n$-level computations.} for estimating $X_L$ in Lemma \ref{lem:xlgammafactors}. Unfortunately for our applications it also requires $\supp(\hphi) \subset (-1/4, 1/4)$. In the proof above $w$ was a function of $k$; this was deadly as we had a factor of $R^{\sigma w} = k^{2\sigma w}$ from $\hphi$ from the contour shift (see \eqref{eq:boundphicontourshiftw}). This forced us to take $\sigma < 1/4$, as our denominator was (essentially) $k^{w/2}$. We sketch an alternative approach using H$\ddot{{\rm o}}$lder's inequality; unfortunately this method also forces $\sigma < 1/4$ (and gives a worse error term).

We apply \eqref{eq:betaidentity} with $a = \frac{1-w}4+\frac{k-1}2$ and $b-a = \frac{w}2$. We choose $w = 1/8$ for definiteness and ease of exposition; similar results hold for all $w$, always requiring $\sigma < 1/4$. For such $w$, we have $b-a-1 < 0$; thus the factor of $(1-t)^{b-a-1}$ is very large for $t$ near 1. We surmount this by using H$\ddot{{\rm o}}$lder's inequality, which states that if $p, q \ge 1$ with $1/p + 1/q = 1$ then \be \int_0^1 |f(t)g(t)|dt \ \le \ \left(\int_0^1 |f(t)|^pdt\right)^{1/p} \cdot \left(\int_0^1 |g(t)|^qdt\right)^{1/q}. \ee We let $f(t) = t^{a-1}$, $g(t) = (1-t)^{b-a-1}$, $p = \frac{16}{1-16\gep}$ and $q=\frac{16}{15+16\gep}$ in H$\ddot{{\rm o}}$lder's inequality, yielding \be \int_0^1 t^{a-1} (1-t)^{b-a-1}dt \ \le \ \left(\int_0^1 t^{(a-1)p}\right)^{1/p} \cdot \left(\int_0^1 (1-t)^{(b-a-1)q}dt\right)^{1/q}. \ee As $(b-a-1) q = -\frac{15}{15+16\gep} > -1$, the integral involving $1-t$ is just $O(1)$. The integral involving $t$ is $\left((a-1)p+1\right)^{-1/p} \ll k^{-1/16+\gep}$. For $\sigma < 1/4$, the $k^{2\sigma w}$ term from \eqref{eq:boundphicontourshiftw} will be smaller than $k^{1/16}$.
\end{rek}

\begin{lem}\label{lem:decayphi}
Let $\phi$ be an even Schwartz function such that ${\rm supp}(\hphi)
\subset (-\sigma,\sigma)$. Then \be \phi(t+iy) \ \ll_{n,\phi} \ e^{2\pi y
\sigma} \cdot (t^2 + y^2)^{-n}. \ee
\end{lem}

\begin{proof} From the Fourier inversion formula, integrating by parts and the compact support of $\hphi$,
we have \bea \phi(t+iy) & \ = \ & \int_{-\infty}^\infty \hphi(\xi)
e^{2\pi i (t+iy)\xi} d\xi \nonumber\\ &=& \int_{-\infty}^\infty
\hphi^{(2n)}(\xi) \cdot (2\pi i (t+iy))^{-2n} e^{2\pi i
(t-iy)\xi}d\xi \nonumber\\ & \ll & e^{2\pi |y| \sigma}
(t^2+y^2)^{-n}. \eea
\end{proof}

\begin{lem}\label{lem:sizeprodpfnofu} Let \be A(x+iy) \ = \ \prod_p
\left(1 - \frac{p^{x+iy}-1}{p(p^{1+x+iy}+1)}\right). \ee If $x > -1$
then $A(x+iy) = O(1)$. \end{lem}

\begin{proof} We have \be |A(x+iy)| \ \le \ \prod_p\left(1 +
\frac{2\max(1,p^x)}{p^{2+x}}\right), \ee and the product is $O(1)$
as long as $x > -1$. \end{proof}

%%%%%%%%%%%%%%%%%%%%%%%%%%%%%%%%%%%%%%%%%%%%%%%%%%%%%%%%%%%%%%%%%%%%%%%%%%%%%%%%%%%%%%%%%%%%%
%%%%%%%%%%%%%%%%%%%%%%%%%%%%%%%%%%%%%%%%%%%%%%%%%%%%%%%%%%%%%%%%%%%%%%%%%%%%%%%%%%%%%%%%%%%%%

\section{Mertens' theorem and how we extend the sums.}\label{sec:mertenssumextension}

We examine other ways of completing the product of the second factor
in the definition of $R_{H_k^\ast(N)}(\alpha,\gamma)$, and the consequences of
this alternate completion on $R_{H_k^\ast(N)}'(r,r)$. Recall this second factor
contributes the product \bea\label{eq:prodratiosexpandyfixed}
\prod_{p \le y}
\left(1-\frac1{p^{1-\alpha+\gamma}}\right) \cdot
\left(1+\frac{p^{1-\alpha+\gamma}}{p^{1+2\gamma}(p^{1-\alpha+\gamma}-1)}\right)
\cdot \prod_{p > y} \left(1+\frac1{p^{1+2\gamma}}\right); \nonumber\\ \eea we wrote
it this way as we wanted to pull out factors of
$1/\zeta(1-\alpha+\gamma)$ before sending $y\to\infty$. We now
analyze this contribution in another manner. We do not pull out the
factors of $1/\zeta(1-\alpha+\gamma)$, and we keep $y$ fixed and
finite. To find the derivative with respect to $\alpha$ forces us to
analyze the following (we ignore the product over
$p > y$ for now as these terms have no $\alpha$ dependence):
\be\label{eq:alternateproductyfixed} \prod_{p\le y}
\left(1 - \frac{1}{p^{1-\alpha+\gamma}} +
\frac1{p^{1+2\gamma}}\right); \ee here the product over $p \le y$ follows
from brute force multiplication of the two terms in
\eqref{eq:prodratiosexpandyfixed}. Keeping $y$ fixed, we now
calculate the derivative of \eqref{eq:alternateproductyfixed} with
respect to $\alpha$: \bea & & \frac{d}{d\alpha} \left[\prod_{p\le y} \left(1 - \frac{1}{p^{1-\alpha+\gamma}} +
\frac1{p^{1+2\gamma}}\right)\right]\Bigg|_{\alpha=\gamma=r} \nonumber\\
& & = \ \prod_{p\le y} \left(1 -
\frac{1}{p^{1-\alpha+\gamma}} + \frac1{p^{1+2\gamma}}\right) \Bigg|_{\alpha=\gamma=r}
\nonumber\\ & & \ \ \ \ \ \cdot \ \frac{d}{d\alpha}
\log\left[\prod_{q\le y} \left(1 -
\frac{1}{q^{1-\alpha+\gamma}} + \frac1{q^{1+2\gamma}}\right)
\right]\Bigg|_{\alpha=\gamma=r} \nonumber\\ & & = \ \prod_{p \le y} \left(1-\frac1p+\frac1{p^{1+2r}}\right) \cdot
\sum_{q \le y} \left(1 - \frac{1}{q} +
\frac1{q^{1+2r}}\right)^{-1} \cdot \frac{-\log q}{q}.\ \ \ \ \ \ \eea It is here
that we must be careful in how we complete the sums (i.e., in how we
let $y\to\infty$). For $\Re(r) > 0$ we write \be \prod_{p \le y} \left(1-\frac1p+\frac1{p^{1+2r}}\right) \ =  \
\prod_{p \le y} \left(1-\frac1p\right) \cdot \prod_{p
\le y} \left(1+\frac1{(p-1)p^{2r}}\right); \ee as
$\Re(r) > 0$ the second factor is of size $1$. By Mertens' Theorem
we have \bea \prod_{p \le y} \left(1-\frac1p\right) \
= \ \frac{e^{-\gamma}}{\log y} \left(1 + O\left(\frac1{\log
y}\right)\right). \eea Thus this product over primes tries to make
our resulting term small; it is, however, balanced by the sum over
$q$ of $\log q/q$, as \be \sum_{q \le y} \frac{\log
q}{q} \ \sim \ \log y + O(1). \ee Completing the book-keeping, we
find a very similar result for the second term in Lemma
\ref{lem:Rprimealphagammar}. Sending $y\to\infty$ gives us the
second term in Lemma \ref{lem:Rprimealphagammar} but now multiplied
by $e^{-\gamma}$.

This is a fascinating observation. It shows that there are at least
\emph{two} natural answers, and their main terms differ by $e^{-\gamma}$. Which is correct? It will almost surely be impossible to tell, as this term contributes $O(1/N)$, and thus is \emph{well} beyond current technology!

Moreover, there is a lot of number theory and probability tied up in
$e^{-\gamma}$. Instead of the prime numbers, one could instead look
at `random' primes. There are many different models one can use to
generate sequences of `random' primes. In the most natural, the
Riemann hypothesis is true with probability one; however, here by RH
we mean $\pi(x) = {\rm Li}(x) + O(x^{1/2+\gep})$. In sieving
heuristics, the number of primes at most $x$ is about $2e^{-\gamma}
x /\log x$, where $2e^{-\gamma} \approx 1.12292$. It
is fascinating that the difference is equivalent to the differences
in viewing the primes as random independent events versus including
the congruence relations! See \cite{BK,Ha,HW,Gr,NW,Wu} for additional remarks on $e^{-\gamma}$.

%%%%%%%%%%%%%%%%%%%%%%%%%%%%%%%%%%%%%%%%%%%%%%%%%%%%%%%%%%%%%%%%%%%%%%%%%%%%%%%%%%%%%%%%%%%%%
%%%%%%%%%%%%%%%%%%%%%%%%%%%%%%%%%%%%%%%%%%%%%%%%%%%%%%%%%%%%%%%%%%%%%%%%%%%%%%%%%%%%%%%%%%%%%

\ \\

\end{document}